%% file: driver.tex
\documentclass[a4paper]{amsart}

\usepackage[lite]{amsrefs}
\usepackage{amssymb}              
\usepackage{longtable}
\usepackage[all]{xy}
\usepackage{wasysym}
\CompileMatrices

\newtheorem{theorem}{Theorem}[section]
\newtheorem{proposition}[theorem]{Proposition}
\newtheorem{lemma}[theorem]{Lemma}

\newtheorem{classlist}[theorem]{Classification}

\theoremstyle{definition}
\newtheorem{definition}[theorem]{Definition}
\newtheorem{construction}[theorem]{Construction}
\newtheorem{example}[theorem]{Example}
\newtheorem{remark}[theorem]{Remark}
\newtheorem{algorithm}[theorem]{Algorithm}

\newcommand{\KK}{\mathbb{K}}

\newcommand{\QQ}{\mathbb{Q}}
\newcommand{\TT}{\mathbb{T}}
\newcommand{\ZZ}{\mathbb{Z}}

\DeclareMathOperator{\Cl}{\mathrm{Cl}}
\DeclareMathOperator{\cone}{\mathrm{cone}}
\DeclareMathOperator{\conv}{\mathrm{conv}}
\DeclareMathOperator{\im}{\mathrm{im}}
\DeclareMathOperator{\lcm}{\mathrm{lcm}}
\DeclareMathOperator{\LP}{\mathrm{LP}}

\DeclareMathOperator{\Spec}{\mathrm{Spec}}

\newcounter{idtci}
\newcounter{tabits}

\title[On Gorenstein Fano toric complete intersections]{On Gorenstein Fano toric complete intersections}

\author[J\"urgen Hausen, Paul Wei\ss]{J\"urgen Hausen, Paul Wei\ss}

\address{Mathematisches Institut, Universit\"at T\"ubingen,
Auf der Morgenstelle 10, 72076 T\"ubingen, Germany}
\email{juergen.hausen@uni-tuebingen.de}

\address{Mathematisches Institut, Universit\"at T\"ubingen,
Auf der Morgenstelle 10, 72076 T\"ubingen, Germany}
\email{paul.weiss@uni-tuebingen.de}


\begin{document}

\begin{abstract}
We classify $\QQ$-factorial Gorenstein Fano 
non-degenerate complete intersection threefolds in
fake weighted projective spaces.
\end{abstract}

\maketitle

\input{intro.tex}


\input{background.tex}


\input{weight-vectors.tex}


\input{classification-lists.tex}


\input{references.tex}
\end{document}

%% file: intro.tex
\section{Introduction}

We contribute to the study of log terminal Fano threefolds
of Picard number one. Our main result adds in particular the missing
piece ``78'' to the following table, listing the numbers of families
of such Fano threefolds in the settings of toric varieties,
varieties with torus action of complexity one (cpl1), full intrinsic
quadrics (fiq) and general non-degenerate toric complete intersections 
(gtci) of rank one:

\medskip

\begin{center}
  
\begin{tabular}{l|c|c|c|c}
& toric & cpl1  & fiq & gtci  
\\[2pt]
\hline
&&&\\[-10pt]
terminal   & \textbf{8}, cf.~\cite{Ka1} & \textbf{47}, cf.~\cite{BeHaHuNi} & \textbf{4}, cf.~\cite{Hi} & \textbf{42}, cf.~\cite{HaMaWr}
\\[2pt]
\hline
&&&\\[-10pt]
Gorenstein & \textbf{48}, cf.~\cite{Ka2} & \textbf{538}, cf.~\cite{BaeHa} & \textbf{11}, cf.~\cite{Hi} & \textbf{78}
\end{tabular}

\end{center}

\medskip

Recall from~\cite{Kh} that a \emph{non-degenerate toric complete intersection}
is a variety that arises from a non-degenerate system $F = (f_1,\ldots,f_c)$ of
Laurent polynomials in~$n$ variables by taking the closure of the zero set
$V(F) \subseteq \TT^n$ inside a toric variety, the fan of which refines the
normal fan of the Minkowski sum of the Newton polytopes of the $f_i$.
A \emph{general toric complete intersection (gtci)} is a family of
non-degenerate toric complete intersections given by an open subset in the
coefficient space of the involved Laurent polynomials.
We say that the gtci is of \emph{rank one} if its ambient toric
variety is $\QQ$-factorial and of Picard number one; in this case,
all members of a suitable gtci are of Picard number one.

\begin{theorem}
Up to isomorphism there are 78 $\QQ$-factorial
Gorenstein Fano general toric complete intersection
threefolds of rank one:
\begin{enumerate}
\item
59 of codimension 1 in a 4-dimensional fake weighted projective
space,
\item
16 of codimension 2 in a 5-dimensional fake weighted projective
space,
\item
3 of codimension 3 in a 6-dimensional fake weighted projective
space.
\end{enumerate}
Each of these families is determined by its degree data and these
are explicitly given in the Classification lists~\ref{thm:classif-w11111}
to~\ref{thm:classif-w1111111}.
\end{theorem}

\tableofcontents

%% file: background.tex
\section{Background and toolkit}

We provide the necessary background, fix our notation and
present first results needed for the classification.
For instance in Construction~\ref{constr:downgrading-gtci}
and Proposition~\ref{prop:downgrading}, we introduce and
study the downgrading of a general toric complete
intersection, which allows us to bound the torsion of the
possible divisor class groups, see Proposition~\ref{bem:bounddtors}
and Algorithm~\ref{algo:torsion}.

Let us begin by discussing the
\emph{fake weighted projective spaces (fwps)},
that means the $\QQ$-factorial projective toric varieties of
Picard number one; see~\cite[Sec.~2]{HaHaHaSp} for an
elementary introduction.
Each $n$-dimensional fwps is encoded by an
$n \times (n+1)$ \emph{generator matrix},
that means a matrix
\[
P
\ = \
\left[
\begin{array}{ccc}
v_1 & \ldots & v_{n+1}
\end{array}
\right],
\]
the columns $v_i \in \ZZ^n$ of which are pairwise
distinct primitive vectors generating $\QQ^n$ as a
convex cone.
The fwps associated with $P$ is the toric variety
$Z(P)$ given by the fan $\Sigma(P)$ with the maximal
cones
\[
\sigma_i \ := \ \cone(v_j; \ j \ne i).
\]
Consider the transpose $P^*$ of $P$, the factor group
$K := \ZZ^{n+1} / \im(P^*)$ and the projection
$Q \colon \ZZ^{n+1} \to K$.
Then the divisor class group and the Cox ring of the
fwps $Z = Z(P)$ are given by
\[
\Cl(Z) = K,
\qquad
\mathcal{R}(Z) = \KK[T_1,\ldots,T_{n+1}],
\qquad
\deg(T_i) = Q(e_i).
\]
Observe that we can write $K \cong \ZZ \oplus \Gamma$
with a finite abelian group $\Gamma$, which in turn
can be written as a direct product of cyclic groups.
Given such a presentation we also view
$Q \colon \ZZ^{n+1} \to K$ as a
\emph{degree matrix in $K$}, that means
\[
Q
\ = \ 
\left[
\begin{array}{ccc}
q_1 & \ldots & q_{n+1}
\end{array}
\right]
\ = \ 
\left[
\begin{array}{ccc}
w_1 & \ldots & w_{n+1}
\\              
\eta_1 & \ldots & \eta_{n+1}
\end{array}
\right],
\quad
w_i \in \ZZ_{>0}, \ \eta_i \in \Gamma,
\]
where any $n$ of the $q_i = (w_i,\eta_i) = Q(e_i)$ generate $K$
as a group.
Note that $Z = Z(P)$ can as well be constructed from
the degree matrix:
the quasitorus $H = \Spec \, \KK[K]$ acts on $\KK^{n+1}$
via the grading $\deg(T_i) = q_i$ and we have 
\[
Z
\ \cong \
Z(Q)
\ := \
\left(\KK^{n+1} \setminus \{0\}\right) / H.
\]

\begin{example}
\label{ex:runningex1}
We consider the fake weighted projective space $Z = Z(P)$ 
with the following generator matrix $P$ and associated
degree matrix $Q$ in $\ZZ \times \ZZ/2\ZZ$:
\[
P
\ = \
\left[
\begin{array}{ccccc}
1 & 1 & 1 & 0 & -1
\\
0 & 3 & 0 & 1 & -2
\\
0 & 0 & 2 & 1 & -2
\\
0 & 0 & 0 & 2 & -2
\end{array}
\right],
\qquad
Q
\ = \
\left[
\begin{array}{ccccc}
1 & 2 & 3 & 6 & 6
\\
\bar 0 & \bar 0 & \bar 1 & \bar 0 & \bar 1
\end{array}
\right].
\]
Then the divisor class group of $Z$ is $\Cl(Z) = \ZZ \times \ZZ/2\ZZ$
and the associated quasitorus is $H = \KK^* \times \{\pm 1\}$, acting
on $\KK^5$ via
\[
(t,-1) \cdot z \ := \ [t z_1, t^2 z_2, -t^3 z_3, t^6 z_4,  -t^6 z_5].
\]  
\end{example}

For classification purposes it is important to decide in an
efficient manner if given defining data define isomorphic fake
weighted projective spaces. The following approach in terms of
degree matrices, due to Ghirlanda~\cite{Ghi}, turns out
to be suitable for our needs.

\begin{remark}
\label{rem:Qadmop}
Let $K = \ZZ \times \Gamma$ be an abelian group with
$\Gamma$ finite.
Any automorphism of the group $K$ inducing the identity
on $\ZZ$ turns degree matrices into degree matrices.
In particular, when dealing with degree matrices in $K$,
we may assume the torsion part $\Gamma$ to be in
\emph{standard form}:
\[
\Gamma \ = \ \ZZ/n_1\ZZ \times \ldots \times \ZZ/n_k\ZZ,
\qquad n_k \mid n_{k-1}, \ldots , n_2 \mid n_1.
\]
If a degree matrix $Q$ in $K$ has standard form, then
one has the following three types of well-defined and
invertible elementary operations on its $i$-th row:
\[
\begin{array}{ll}
\zeta Q_{i*}, & \zeta \in (\ZZ/n_i\ZZ)^*, \ 2 \le i \le 1+k,
\\[3pt]
Q_{i*} + Q_{j*}, & 1 \le j < i \le 1+k,
\\[3pt]
Q_{i*} + \frac{n_j}{n_i}Q_{j*}, & 2 \le i < j \le 1+k.
\end{array}
\]
According to~\cite{Ghi}, two degree matrices define isomorphic
fwps if and only if they can be transformed into each other
via these row operations and swapping columns.
\end{remark}

\goodbreak

We recall from~\cite{Kh,HaMaWr} the basic notation and concepts on
non-degenerate toric complete intersections, adapted to the special
case of fake weighted projective spaces.
Denote by  $\LP(n)$ the algebra of Laurent polynomials in
the variables $T_1, \ldots, T_n$.
The elements of $\LP(n)$ are written as 
$f = \sum \alpha_\nu T^\nu$, where we denote
$T^\nu = T_1^{\nu_1} \cdots T_n^{\nu_n}$
for $\nu \in \ZZ^n$. The \emph{Newton polytope} of such an $f$ is 
$$
B(f)
\ := \
\conv(\nu \in \ZZ^n; \ \alpha_\nu \ne 0)
\ \subseteq \
\QQ^n.
$$
Given a face $B \preccurlyeq B(f)$ of the Newton polytope of $f$,
one defines the associated \emph{face polynomial} as 
$$
f_B
\ = \
\sum_{\nu \in B \cap \ZZ^n} \alpha_\nu T^\nu
\ \in \
\LP(n).
$$
Now, consider a system $F = (f_1, \ldots, f_c)$  of 
Laurent polynomials $f_1, \ldots, f_c \in \LP(n)$.
With $B_j := B(f_j)$, the \emph{Newton polytope of $F$}
is the Minkowski sum
$$
B \ := \ B(F) \ := \ B_1 + \ldots + B_c \ \subseteq \ \QQ^n.
$$ 
The \emph{face system} $F'$ of $F$ associated with a face
$B' \preccurlyeq B$ of the Newton polytope $B$ of $F$ is
the Laurent system
$$
F' \ := \ (f_1', \ldots, f_c'),
$$
where $f_j' = f_{B_j'}$ are the face polynomials
associated with the faces $B_j' \preccurlyeq B_j$
from the (unique) presentation
$B' = B_1' + \ldots + B_c'$ with
$B_j' \preccurlyeq B_j$.

Let $B_1, \ldots, B_s \subseteq \QQ^n$ be 
$n$-dimensional lattice polytopes.
Then each of the $B_j$ defines a finite-dimensional
vector subspace
$$ 
U(B_j)
\ := \ 
\bigoplus_{\nu \in B_j \cap \ZZ^n} \KK \cdot T^\nu
\ \subseteq \
\LP(n).
$$
The \emph{Laurent space} associated with
the polytopes $B_1, \ldots, B_s \subseteq \QQ^n$ 
is the finite-dimensional vector space
$$
U(B_1, \ldots, B_s)
\ := \
U(B_1) \oplus \ldots \oplus U(B_s).
$$
By a \emph{general Laurent system}
we mean a non-empty open subset 
$\mathcal{F} \subseteq U(B_1, \ldots, B_s)$
such that every 
$F = (f_1,\ldots,f_s) \in \mathcal{F}$ satisfies
$$ 
B(f_1) = B_1, \ \ldots, \ B(f_s) = B_s. 
$$

\begin{definition}
\label{def:nondeg}
A system $F = (f_1, \ldots, f_c)$ of Laurent polynomials
in $\LP(n)$ is called \emph{non-degenerate} if for every
face $B' \preccurlyeq B$, the Jacobian
$$ 
J_{F'}(z)
\ = \ 
\left( 
\frac{\partial f_j'}{\partial z_k} (z)
\right)_{\genfrac{}{}{0pt}{}{1 \le j \le c}{1 \le k \le n}} 
$$
of the associated face system is of rank $c$ for all
$z \in V(F') \subseteq \TT^n$.
A general Laurent system $\mathcal{F}$ is
\emph{non-degenerate} if 
every Laurent system $F \in \mathcal{F}$ is so.
\end{definition}

\begin{remark}
\label{rem:glsexist}
From~\cite[Sec.~2]{Kh} we infer that every
collection $B_1, \ldots, B_c \subseteq \QQ^n$ 
of $n$-dimensional lattice polytopes admits
a general non-degenerate Laurent system.
\end{remark}

\begin{construction}
\label{constr:gtci}
Let $Z=Z(P)$ be an $n$-dimensional fwps,
$B_1,\ldots,B_c$, where $c < n$, be
lattice polytopes of dimension $n$
such that $B := B_1 + \ldots + B_c$
has $\Sigma(P)$ as its normal fan and
$\mathcal{F} \subseteq U(B_1, \ldots, B_c)$
a general non-degenerate Laurent system.
For each $F \in \mathcal{F}$ set 
\[
X(F) \ := \ \overline{V(F)} \ \subseteq \ Z,
\]
where $V(F) \subseteq \TT^n$ denotes the zero set
of $F$ in the torus $\TT^n \subseteq Z$.
Then $X(F)$ is an irreducible, normal, projective 
variety of dimension $d = n-c$,
see~\cite[Cor.~3.10 and Thm.~4.3]{HaMaWr}.
We call the set $\mathcal{X}$ of all $X(F)$, where
$F \in \mathcal{F}$, a
\emph{general toric complete intersection (gtci)} of type
$(d,c)$ in $Z$ associated with $B_1,\ldots,B_c$.
\end{construction}

\begin{example}
\label{ex:runningex2}
We present a general toric complete intersection in the
fake weighted projective space $Z=Z(P)$
from Example~\ref{ex:runningex1}.
Consider
\[
f(T_1,T_2,T_3,T_4)
\ := \
1
+
\frac{T_{2}^{4}T_{3}^{6}}{T_{1}^{12}T_{4}^{5}}
+
\frac{T_{2}^{6}T_{3}^{6}}{T_{1}^{12}T_{4}^{6}}
+
\frac{T_{2}^{4}T_{3}^{8}}{T_{1}^{12}T_{4}^{6}}
+
\frac{T_{2}^{4}T_{3}^{6}}{T_{1}^{12}T_{4}^{4}}.
\]
One directly verifies that the Laurent polynomial $f$ is non-degenerate
in the sense of Definition~\ref{def:nondeg}.
The vertices of the Newton polytope $B$ of $f$ are
\[
(0, 0, 0, 0), \
(-12, 4, 6, -5), \
(-12, 6, 6, -6), \
(-12, 4, 8, -6), \
(-12, 4, 6, -4).
\]
A direct computation shows that the normal fan of $B$ has the columns of
$P$ from Example~\ref{ex:runningex1} as its primitive ray generators.
For the coefficient space, we have
\[
\dim(U(B))
\ = \ 
\left| B \cap \ZZ^4 \right|
\ = \
21.
\]
As non-degeneracy is an open condition, there is a gtci
$\mathcal{X}$ in $Z$ containing $X(f)$ for our initial
Laurent polynomial $f$.
\end{example}

\begin{construction}
Let $Z = Z(P)$ be an $n$-dimensional fwps, $f \in \LP(n)$
and set $r := n+1$.
The pullback of~$f$ with respect
to the homomorphism $p \colon \TT^r \to \TT^n$
defined by $P = (p_{ij})$ has a unique presentation
as
$$
p^*f(T_1, \ldots, T_n)
\ = \
f(T_1^{p_{11}} \cdots T_r^{p_{1r}}, \ldots, T_1^{p_{n1}} \cdots T_r^{p_{nr}})
\ = \
T^{\nu}g(T_1, \ldots, T_r)
$$
with $T^{\nu} \in \LP(r)$
and a $K$-homogeneous polynomial
$g \in \KK[T_1,\ldots,T_r]$ being
coprime to each of $T_1, \ldots, T_r$.
We call~$g$ the \emph{$P$-homogenization}
of $f$.
\end{construction}

\begin{example}
\label{ex:runningex3}
Consider the setting of Examples~\ref{ex:runningex1}
and~\ref{ex:runningex2}.
Then, for the $P$-homogenization $g$ of the Laurent
polynomial $f$, we have
\[
g
\ = \
T_1^{12} + T_2^{6} + T_3^{4} + T_4^{2} + T_5^{2},
\qquad
\deg(g) \ = \ (12,\bar 0) \ \in \ \ZZ \oplus \ZZ/2\ZZ.
\]
\end{example}

\begin{proposition}
\label{prop:gtciprops1}
Consider an fwps $Z = Z(P)$ and polytopes
$B_1,\ldots,B_c$ as in Construction~\ref{constr:gtci}
and assume $d := n-c  \ge 3$.
Then there is an associated gtci $\mathcal{X}$ of
type $(d,c)$ in $Z$ satisfying the following.
Any $X \in \mathcal{X}$ has divisor class group
\[
\Cl(X) \ = \ \Cl(Z).
\]
If $X \in \mathcal{X}$ is given by the
Laurent system $(f_1,\ldots,f_c)$ and $g_j$
denotes the $P$-homogenization of $f_j$,
then the Cox ring of $X$ is the $\Cl(X)$-graded
factor ring
\[
\mathcal{R}(X)
\ = \
\mathcal{R}(Z) / \langle g_1,\ldots,g_c \rangle
\ = \
\KK[T_1, \ldots, T_{n+1}] / \langle g_1,\ldots,g_c \rangle,
\]  
Moreover, the anticanonical divisor class of $X$ is
the sum of the generator degrees minus
the sum of the relation degrees:
\[
-\mathcal{K}_X  
\ = \
\deg(T_1) + \ldots + \deg(T_{n+1}) - \deg(g_1) - \ldots - \deg(g_c).
\]
With the $\ZZ$-part $k$ of $-\mathcal{K}_X$, 
the order $\tau$ of the torsion part of $\Cl(X)$,
the $\ZZ$-parts~$w_i$ of $\deg(T_i)$ and the $\ZZ$-parts
$m_j$ of $g_j$, we have 
\[
\left(- \mathcal{K}_X\right)^d
\ = \
\frac{m_1 \cdots m_c}{w_1 \cdots w_{d+c+1}} \cdot \frac{k^d}{\tau}.
\]
Finally, each Cox ring relation $g_j$ has 
a power of any $T_1, \ldots, T_{n+1}$ as a monomial,
that means that it is of the form 
\[
g_j  
\ = \
a_{j,1}T_1^{l_{j,1}} + \ldots + a_{j,n+1}T_{n+1}^{l_{j,n+1}}
\, + \,
\sum a_{j,\nu} T^\nu,
\quad
a_{j,1}, \ldots, a_{j,n+1} \ne 0,
\]
where the $T^\nu$ are monomials in at least two variables.
In particular, $X$ is Fano if and only if the $\ZZ$-parts
$w_i$ of $\deg(T_i) \in \Cl(X)$ satisfy
\[
w_1 + \ldots + w_r \ > \ (l_{1,1} + \ldots + l_{c,1})w_1.
\]
\end{proposition}

\begin{proof}
Remark~\ref{rem:glsexist} tells us that the polytopes
$B_1,\ldots,B_c$ admit a gtci $\mathcal{X}$.
After suitably shrinking $\mathcal{X}$, we can
apply~\cite[Cor.~4.20]{HaMaWr} to obtain the statements
about the divisor class group and the Cox ring.
Moreover, \cite[Prop.~4.13~(i)]{HaMaWr} yields the
presentation of the anticanonical divisor class
and the formula for the anticanonical self
intersection follows
from~\cite[Lemma~2.1.4.1, Prop.~2.4.2.11, Constr.~3.3.3.4]{ArDeHaLa}.
Finally~\cite[Lemma~5.1]{HaMaWr} shows that
each $g_j$ hosts a power of any variable as a
monomial.
\end{proof}

\begin{proposition}
\label{prop:h0KX}
Let $Z=Z(P)$ be an $n$-dimensional fwps,
fix an associated degree matrix $Q$ and
let $\mathcal{X}$ be a gtci of type $(d,c)$
in $Z$ as in Proposition~\ref{prop:gtciprops1}.
Set
\[
e_X \ := \ \sum_{i=1}^{n+1} e_i - \sum_{j=1}^c l_{j,1}e_1.
\]
Then $Q(e_X)$ equals the anticanonical class
$-\mathcal{K}_X$ in $\Cl(Z)$ and we obtain convex
polytopes in $\QQ^n$ by
\[
B(-\mathcal{K})
 := 
(P^*)^{-1}\left(\QQ^{n+1}_{\ge 0}-e_X\right),
\qquad
C_j       
 := 
(P^*)^{-1}\left(\QQ^{n+1}_{\ge 0}-e_X+l_{j,1}e_1\right).
\]
With the respective numbers $s(-\mathcal{K})$ and
$s_j$ of lattice points of these polytopes,
the dimension of the section space of any
anticanonical divisor on $X$ is given as
\[
h^0(-\mathcal{K}_X) \ = \ \max(s(-\mathcal{K}) - s_1 - \ldots - s_c, \, 0).
\]
\end{proposition}

\begin{proof}
The section space of an anticanonical divisor $-D_X$
is isomorphic to the homogeneous component of degree
$-\mathcal{K}_X$ of the Cox ring of $X$, which in turn is
given as
\[
H^0(X,-D_X)
\ \cong \  
\KK[T_1, \ldots, T_{n+1}]_{-\mathcal{K}_X} / \langle g_1, \ldots, g_c \rangle_{-\mathcal{K}_X},
\]  
with the $P$-homogenizations $g_1, \ldots, g_c$ of the defining
Laurent polynomials as in Proposition~\ref{prop:gtciprops1}.
Now, $s(-\mathcal{K})$ equals the dimension of 
$\KK[T_1, \ldots, T_{n+1}]_{-\mathcal{K}_X}$ and $s_j$
counts the monomial multiples
of $g_j$.
\end{proof}

\begin{proposition}
\label{prop:gtciprops2}
Let $Z=Z(P)$ be an $n$-dimensional fwps, write $\Cl(Z) = \ZZ \times \Gamma$
and fix an associated degree matrix
\[
Q
\ = \
\left[  
\begin{array}{ccc}
q_1 & \ldots & q_{n+1}
\end{array}
\right]
\ = \
\left[  
\begin{array}{ccc}
w_1 & \ldots & w_{n+1}
\\
\eta_1 & \ldots & \eta_{n+1}  
\end{array}
\right],
\qquad
w_i \in \ZZ_{>0},
\quad
\eta_i \in \Gamma.
\]
Consider a general toric complete intersection $\mathcal{X}$
of type $(d,c)$ in $Z$ with $d \ge 3$ and let $X \in \mathcal{X}$.
\begin{enumerate}
\item
$X$ is Gorenstein iff
for any 
$1 \le i_1 < \ldots < i_{1+c} \le n+1$
and the complementary 
$1 \le j_1 < \ldots < j_d \le n+1$,
we have 
\[
q_{j_1} + \ldots + q_{j_d}
\ \in \
\ZZ q_{i_1} + \ldots + \ZZ q_{i_{1+c}}.
\]
\item
If $\Gamma=0$, then $X$
is Gorenstein iff for all
$1 \le i_1 < \ldots < i_{1+c} \le n+1$
and the complementary 
$1 \le j_1 < \ldots < j_d \le n+1$,
we have 
\[
\gcd(w_{i_1}, \ldots, w_{i_{1+c}})
\mid
(w_{j_1} + \ldots + w_{j_d}).
\]
\end{enumerate}
\end{proposition}

\begin{proof}
We prove~(i).
The displayed condition in the first assertion
characterizes the fact that the canonical
divisor $\mathcal{K}_Z$ of the ambient
toric variety $Z$ is Cartier
near the point $[z]$ with the coordinates
$z_{j_1}, \ldots, z_{j_d} \ne 0$ and 
$z_{i_1} = \ldots = z_{i_{1+c}} = 0$,
see~\cite[Prop.~2.4.2.3]{ArDeHaLa}.
By~\cite[Thm.~4.8 and Cor.~4.9]{HaMaWr},
these points $[z]$ yield exactly the
closed orbits of the minimal open toric
subvariety $Z_X \subseteq Z$ containing $X$.
Thus, the condition of~(i) encodes the fact
that $\mathcal{K}_Z$ is Cartier on $Z_X$.
By~\cite[Prop.~4.13~(ii)]{HaMaWr}, the latter
is equivalent to $\mathcal{K}_X$ being
Cartier.
The second assertion is a simple special case
of the first one.
\end{proof}

\begin{construction}
\label{constr:downgrading-degmat}
Let $Q = [q_1,\ldots, q_{n+1}]$ be a degree matrix
in $K = \ZZ \times \Gamma$ and write $q_i = (w_i,\eta_i)$
accordingly.
Any epimorphism $\psi \colon \Gamma \to \tilde \Gamma$
of abelian groups gives rise to degree matrix
in $K = \ZZ \times \tilde \Gamma$, the \emph{downgrading}
\[
\tilde Q
\ := \
\left[  
\begin{array}{ccc}
\tilde q_1 & \ldots & \tilde q_{n+1}
\end{array}
\right],
\qquad
\tilde q_i := (w_i, \tilde \eta_i),
\quad
\tilde \eta_i := \psi(\eta_i).
\]
\end{construction}

\begin{construction}
\label{constr:downgrading-gtci}
Consider an $n$-dimensional fwps $Z=Z(P)$ with degree
matrix $Q$ in $K = \ZZ \times \Gamma$ and a downgrading
$\tilde Q$ in $\tilde K = \ZZ \times \tilde \Gamma$
given by $\psi \colon \Gamma \to \tilde \Gamma$.
Then we have a commutative diagram of homomorphisms
\[
\xymatrix{
0
\ar@{<-}[r]
&    
K
\ar@{<-}[r]^Q
\ar[d]_{\psi}
&    
{\ZZ^{n+1}}
\ar@{<-}[r]^{\quad P^*}
\ar@{=}[d]
&
{\ZZ^n}
\ar@{<-}[r]  
\ar[d]^{A}
&
0
\\
0
\ar@{<-}[r]  
&    
{\tilde K}
\ar@{<-}[r]_{\tilde Q}
&    
{\ZZ^{n+1}}
\ar@{<-}[r]_{\quad \tilde P^*}
&
{\ZZ^n}
\ar@{<-}[r]  
&
0
}
\]
where $\tilde P$ is any generator matrix associated with
$\tilde Q$ and the matrix $A$ is uniquely determined by
the choice of $\tilde P$.
Let $\mathcal{X}$ be a gtci of type $(d,c)$ in $Z$
given by a general Laurent system $\mathcal{F}$ with
Newton polytopes $B_1, \ldots, B_c$. Set
\[
\tilde B_i \ := \ A \cdot B_i, \qquad i = 1,\ldots,c.
\]
Then the polytopes $\tilde B_i$ admit a gtci
$\tilde{\mathcal{X}}$ in the fwps $\tilde Z$ with
generator matrix $\tilde P$ such that every
$\tilde f \in \tilde{\mathcal{F}}$ specializes to an 
$\alpha(f)$, where $f \in \mathcal{F}$, and $\alpha$ sends a
monomial $T^\nu$ to $T^{A \cdot \nu}$.
We call $\tilde{\mathcal{X}}$ a \emph{downgrading}
of $\mathcal{X}$  
given by $\psi \colon \Gamma \to \tilde \Gamma$.
\end{construction}


\begin{example}
\label{ex:runningex3}
In the situation of Examples~\ref{ex:runningex1}
and~\ref{ex:runningex2}, we can project the group
$K = \ZZ \times \ZZ/2\ZZ$ onto its first factor
$\tilde K = \ZZ$, which provides the downgrading
\[
\tilde Q
\ = \ 
\left[
\begin{array}{ccccc}
1 & 2 & 3 & 6 & 6
\end{array}
\right]
\]
in $\tilde K$ of the degree matrix $Q$ in $Z$.
The toric variety $\tilde Z$ with degree matrix
$\tilde Q$ is an ordinary weighted projective
space. Moreover, with
\[
\tilde P
\ = \
\left[
\begin{array}{ccccc}
1 & 1 & 1 & 0 & -1
\\
0 & 3 & 0 & 0 & -1
\\
0 & 0 & 2 & 0 & -1
\\
0 & 0 & 0 & 1 & -1
\end{array}
\right],
\qquad
A
\ = \
\left[
\begin{array}{cccc}
1 & 0 & 0 & 0 
\\
0 & 1 & 0 & 0 
\\
0 & 0 & 1 & 0
\\
0 & 1 & 1 & 2
\end{array}
\right],
\]
we have a generator matrix $\tilde P$ associated with $\tilde Q$
and $\tilde P^* \cdot A = P$.
The vertices of the polytope $\tilde B = A \cdot B$ are given by
\[
(0, 0, 0, 0), \
(-12, 4, 6, 0), \
(-12, 6, 6, 0), \
(-12, 4, 8, 0), \
(-12, 4, 6, 2).
\]
The associated general Laurent system $\tilde{\mathcal{F}}$
is of dimension $| \tilde B \cap \ZZ^4 | = 36$,
whereas $\alpha(\mathcal{F})$ with $\mathcal{F}$ from 
Example~\ref{ex:runningex2} is of dimension 21.
\end{example}

\begin{proposition}
\label{prop:downgrading}
Let $Z=Z(P)$ be an $n$-dimensional fake weighted
projective space, write $\Cl(Z) = \ZZ \times \Gamma$
and fix an associated degree matrix
\[
Q
\ = \
\left[  
\begin{array}{ccc}
w_1 & \ldots & w_{n+1}
\\
\eta_1 & \ldots & \eta_{n+1}  
\end{array}
\right],
\qquad
w_i \in \ZZ_{>0},
\quad
\eta_i \in \Gamma.
\]
Let $\mathcal{X}$ be a general tci of type $(d,c)$ in $Z$
with $d \ge 3$ and $X \in \mathcal{X}$ given by
$(f_1,\ldots,f_c)$.
Then any subgroup $\Gamma_0 \subseteq \Gamma$
defines a commutative boletus edulis

\bigskip

\bigskip

\[
\xymatrix{
{\bar X}
\ar@{}[r]|\supseteq
\ar@/^1.8pc/[rrr]
&
{\hat X}
\ar@{}[r]|\subseteq
\ar[d]
\ar@/_1.8pc/[dd]_{/H}
&
{\hat Z}
\ar[d]_{/H_0 \hspace*{8pt} }
\ar@/^2pc/[dd]^{/H}
\ar@{}[r]|\subseteq
&
{\bar Z}
\\
&
{\tilde X}
\ar@{}[r]|\subseteq
\ar[d]
&
{\tilde Z}
\ar[d]_{/G \hspace*{10pt} }
&
\\
&
X
\ar@{}[r]|\subseteq
&
Z
&
}
\]

\bigskip

\noindent
with the total coordinate spaces 
$\bar X = V(g_1,\ldots,g_c) \subseteq  \KK^{n+1} = \bar Z$,
where $g_i$ denotes the $P$-homogenization of $f_i$, 
the subsets $\hat X = \bar X \setminus \{0\}$ and 
$\hat Z = \bar Z \setminus \{0\}$, the quasitori
\[
H_0 \ = \ \Spec \, \KK[\ZZ \times \Gamma/\Gamma_0]
\ \subseteq \ 
\Spec \, \KK[\ZZ \times \Gamma] \ = \ H,
\]
the finite abelian factor group $G = H/H_0$ and the fake
weighted projective space $\tilde Z = \hat Z / H_0$,
a degree matrix for which is given by the downgrading
\[
\tilde Q
\ = \
\left[  
\begin{array}{ccc}
w_1 & \ldots & w_{n+1}
\\
\tilde \eta_1 & \ldots & \tilde \eta_{n+1}  
\end{array}
\right],
\qquad
\tilde \eta_i = \eta_i + \Gamma_0 \in \Gamma/\Gamma_0.
\]
Moreover, let $\tilde{\mathcal{X}}$ be the general toric
complete intersection of type $(d,c)$ in $\tilde Z$
obtained from $\mathcal{X}$ via downgrading.
Then the following statements hold:
\begin{enumerate}
\item
$\tilde{\mathcal{X}}$ is Fano if and only if
$\mathcal{X}$ is Fano.
\item
If $\mathcal{X}$ is Gorenstein, then $\tilde{\mathcal{X}}$
is Gorenstein.
\end{enumerate}
\end{proposition}

\begin{proof}
The first assertion is clear due to the characterization of the
Fano property given in Proposition~\ref{prop:gtciprops1}
and, the second one follows from Proposition~\ref{prop:gtciprops2}.
\end{proof}


We turn to the technical framework of the
subsequent classification.
First important concepts are weight vectors,
the prospective first rows of the degree
matrices and their associated exponent vectors.

\begin{definition}
\label{def:wvev}
A \emph{weight vector of type $(d,c)$}
is a vector $w = (w_1, \ldots, w_{1+d+c})$
with positive integer entries such that
\[
w_1 \le \ldots \le w_{1+d+c}
\]  
and $w$ is \emph{almost free}, i.e.
any $d+c$ of the $w_i$ are coprime.
An \emph{exponent vector} for such $w$
is a vector $l = (l_1, \ldots, l_{1+d+c})$
with positive integer entries such that
\[
\mu \ := \ l_1w_1 \ = \ l_2w_2 \ = \ \ldots \ = \ l_{1+d+c}w_{1+d+c}.
\]
We call these equations the \emph{homogeneity conditions}
and $\mu$ the \emph{degree} of $(w,l)$.
An exponent vector $l$
is \emph{true} if $l_{1+d+c} \ge 2$ holds.
\end{definition}

\begin{remark}
\label{rem:llcmdivmu}
Consider a weight vector $w$ of type $(d,c)$,
an exponent vector $l$ for~$w$ and the degree
$\mu$ of $(w,l)$.
Then each of the $l_i,w_j$ divides $\mu$, hence
\[
\begin{array}{rcl}
\lcm(l_1, \ldots, l_{1+d+c}) \mid \mu,
& &
\lcm(w_1, \ldots, w_{1+d+c}) \mid \mu.
\end{array}
\]
Moreover, if a prime number $p$ divides $l_i$,
then $p \mid l_j$ for some $j \ne i$, as otherwise
$\mu=l_jw_j$ would force $p \mid w_j$ for all $j \ne i$. 
\end{remark}

\begin{lemma}
\label{lem:mullcm}
Consider a weight vector $w$ of type $(d,c)$,
an exponent vector $l$ for~$w$ and the degree
$\mu$ of $(w,l)$.
If $\gcd(w_{i_1}, \ldots, w_{i_k}) = 1$ for
$i_1 < \ldots < i_k$, then
\[
\lcm(l_{i_1}, \ldots, l_{i_k}) \ = \ \mu,
\qquad\qquad
w_{i_1} \cdots w_{i_k}  \mid \mu.
\]  
In particular, for any choice $i_1 < \ldots < i_{d+c}$,
the degree $\mu$ equals the least common multiple of
the exponents $l_{i_1}, \ldots, l_{i_{d+c}}$.
\end{lemma}

\begin{proof}
Set $\lambda := \lcm(l_{i_1}, \ldots, l_{i_k})$.
Due to Remark~\ref{rem:llcmdivmu}, we have
$\mu = \lambda \omega$ with a positive integer
$\omega$.
Thus, $l_{i_j}w_{i_j} = \lambda \omega$ 
implies $\omega \mid w_{i_j}$ for $j = 1, \ldots, k$.
Coprimeness of $w_{i_1}, \ldots, w_{i_k}$
implies $\omega=1$.
This shows the first assertion.
The second one follows from the definition of a weight
vector.
\end{proof}

\begin{lemma}
\label{lem:expvequal}
Consider a weight vector $w$ of type $(d,c)$,
a tuple $\ell = (\ell_1, \ldots, \ell_c)$
of exponent vectors for $w$ and the associated
degrees $\mu_1, \ldots, \mu_c$. Then
\[
\mu_i \ell_j \ = \ \mu_j \ell_i
\qquad
\text{for all } 1 \le i,j \le c.
\]
In particular, if for some $1 \le k \le 1+d+c$
we have $\ell_{1,k} = \ldots = \ell_{c,k}$, then
$\mu_1 = \ldots = \mu_c$ and
$\ell_1 = \ldots = \ell_c$ hold.
\end{lemma}

\begin{proof}
The assertion is an immediate consequence of the
homogeneity conditions:
for every $k = 1, \ldots, 1+d+c$, we have 
\[
\mu_i \ell_{j,k}
\ = \
\ell_{i,k} w_k \ell_{j,k}
\ = \ 
\ell_{j,k} w_k \ell_{i,k}
\ = \
\mu_j \ell_{i,k}.
\]
\end{proof}

\begin{definition}
Let $w$ be a weight vector of type $(d,c)$ and
a tuple $\ell = (\ell_1, \ldots, \ell_c)$ of
exponent vectors for $w$.
We say that $(w,\ell)$ is \emph{Fano} if
we have 
\[
w_1 + \ldots + w_{1+d+c}
\ > \ 
(\ell_{1,1}+ \ldots + \ell_{c,1})w_1.
\]
With $\mu_i = \ell_{i,1}w_1$ and $\mu = (\mu_1,\ldots,\mu_c)$,
we call $(w,\mu)$ a \emph{weight-degree constellation of type $(d,c)$}
and say that $(w,\mu)$ is \emph{Fano (true)} if $(w,\ell)$ is Fano
(each $\ell_j$ is true).
\end{definition}

\begin{proposition}
If $(w, \mu)$ is a true Fano weight-degree constellation of
type $(d,c)$, then we have $c \le d$.
\end{proposition}

\begin{proof}
Let $\ell = (\ell_1, \ldots, \ell_c)$ denote the
tuple of exponent vectors for $(w,\mu)$.
The definition of a true Fano weight-degree
constellation yields
\[
w_1 + \ldots + w_{1+d+c}
\ > \ 
(\ell_{1,1+d+c}+ \ldots + \ell_{c,1+d+c})w_{1+d+c}
\ \ge \ 
2cw_{1+d+c},
\]
where the last estimate holds because of $\ell_{i,1+d+c} \ge 2$.
The assertion thus follows from
$w_1 \le \ldots \le w_{1+d+c}$.
\end{proof}

\begin{remark}
\label{rem:altfano}
Let $w$ be a weight vector of type $(d,c)$ and
a tuple $\ell = (\ell_1, \ldots, \ell_c)$ of
exponent vectors for $w$.
Then $(w,\ell)$ is Fano if and only if
\[
\sum_{j=1}^{1+d+c} \frac{1}{\ell_{1,j} + \ldots + \ell_{c,j}}
\ > \ 
1.\]
Note that, given $1 \le i \le d$, we can
rewrite this inequality in terms of the degrees
$\mu_1, \ldots, \mu_c$ associated with $(w,\ell)$
as 
\[
\sum_{j=1}^{1+d+c} \frac{1}{\ell_{i,j}}
\ > \ 
\frac{\mu_1 + \ldots + \mu_{c}}{\mu_i}.\]
Moreover, the first inequality directly implies the following
bound for the sum of the smallest exponents: 
\[
\ell_{1,1+d+c}+ \ldots + \ell_{c,1+d+c}
\ \le \
d+c.
\]
\end{remark}

\begin{definition}
Consider a weight-degree constellation $(w,\mu)$ of type $(d,c)$,
and a finite abelian group $\Gamma$.
By a \emph{degree matrix in $\ZZ \times \Gamma$ associated with $(w,\mu)$}
we mean a matrix
\[
Q
\ = \
\left[  
\begin{array}{ccc}
q_1 & \ldots & q_{1+d+c}
\end{array}
\right]
\ = \
\left[  
\begin{array}{ccc}
w_1 & \ldots & w_{1+d+c}
\\
\eta_1 & \ldots & \eta_{1+d+c}  
\end{array}
\right],
\qquad
\eta_i \in \Gamma,
\]
which is \emph{almost free}, i.e., any $d+c$ of the columns
$q_1, \ldots, q_{1+d+c}$ generate $\ZZ \times \Gamma$ as a group,
and the tuple $\ell$ of exponent vectors of $(w,\mu)$
satisfies $\ell_{j,i} q_i = \ell_{j,k} q_k$ for all
$1 \le j \le c$ and $1 \le i,k \le 1+d+c$.
\end{definition}

\begin{definition}
\label{def:gorenstein}
A weight vector $w$ or a weight-degree constellation $(w,\mu)$
of type $(d,c)$ is \emph{Gorenstein} if for all
$1 \le i_1 < \ldots < i_{1+c} \le 1+d+c$ and the complementary
$1 \le j_1 < \ldots < j_d \le 1+d+c$, we have 
\[
\gcd(w_{i_1}, \ldots, w_{i_{1+c}})
\mid
(w_{j_1} + \ldots + w_{j_d}).
\]
Moreover, we call a degree matrix $Q$ in
$\ZZ \times \Gamma$ associated with $(w,\mu)$
\emph{Gorenstein} if all
$1 \le i_1 < \ldots < i_{1+c} \le 1+d+c$
and their complementary 
$1 \le j_1 < \ldots < j_d \le 1+d+c$
satisfy
\[
q_{j_1} + \ldots + q_{j_d}
\ \in \
\ZZ q_{i_1} + \ldots + \ZZ q_{i_{1+c}}.
\]
\end{definition}

\begin{proposition}
\label{prop:degdata2gtci}
Let $(w,\mu)$ be a weight-degree constellation of type $(d,c)$
with exponent vector tuple $\ell$ and associated degree matrix
$Q$ in $\ZZ \times \Gamma$.
Fix a generator matrix $P$ for $Q$ and set $n := d+c$.
Then we obtain full dimensional lattice polytopes
\[
B_j
\ := \ 
(P^*)^{-1}\left(\QQ^{n+1}_{\ge 0}-\ell_{j,1}e_1\right)
\ \subseteq \
\QQ^n,
\qquad
j = 1, \ldots, c.
\]
The normal fan $\Sigma$ of $B := B_1 + \ldots + B_c$
has the columns of $P$ as its primitive ray generators.
Each member $X$ of the gtci $\mathcal{X}$ in $Z=Z(P)$ defined
by $B_1,\ldots,B_c$ has divisor class group and Cox ring
\[
\Cl(X) = \ZZ \times \Gamma,  
\quad
\mathcal{R}(X) = \KK[T_1,\ldots,T_{n+1}] / \langle g_1,\ldots,g_c \rangle,
\quad
\deg(T_i) = q_i,
\]
where $g_1, \ldots, g_c$ denote the $P$-homogenizations of the
defining Laurent polynomials $f_1,\ldots,f_c \in \LP(n)$ of $X$.
Moreover each $g_j$ has $T_1^{\ell_{j,1}}, \ldots, T_{n+1}^{\ell_{j,n+1}}$
among its monomials.
Finally, $X$ is Fano iff $(w,\mu)$ is so, and $X$ is Gorenstein
iff $Q$ is so.
\end{proposition}

\begin{proof}
We show that each $B_j$ is a full dimensional
lattice polytope.
Because of $\ker(Q) = \im (P^*)$, we have
\[ 
P^*(B_j)
\ = \ 
\ker(Q) \cap \left(\QQ^{n+1}_{\ge 0}-\ell_{j,1}e_1\right).
\]
As $Q$ is a degree matrix for $(w,\mu)$, this implies
$\ell_{j,i}e_i-\ell_{j,i}e_1 \in P^*(B_j)$ for all $i$.
Thus, $P^*(B_j)+\ell_{j,i}e_1$ is the simplex with the
vertices $\ell_{j,i}e_i$. The claim follows.

Now the fact that the normal fan of $B_j$ has the
columns of $P$ as its primitive ray generators
is standard convex geometry,
we directly see that each $g_j$ has
$T_1^{\ell_{j,1}}, \ldots, T_{n+1}^{\ell_{j,n+1}}$
among its monomials and
the rest follows from Propositions~\ref{prop:gtciprops1}
and~\ref{prop:gtciprops2}.
\end{proof}

\begin{proposition}
\label{bem:bounddtors}
Let $w$ be a weight vector of type $(d,c)$ and
$Q$ an associated degree matrix in $\ZZ \times \Gamma$,
where 
\[
\Gamma \ = \ \ZZ/n_1\ZZ \times \ldots \times \ZZ/n_k\ZZ,
\qquad n_k \mid n_{k-1}, \ldots , n_2 \mid n_1.
\]
Then we have $k \le d+c-1$. Moreover, let $l$ be an exponent
vector for $w$ and $\mu = l_1w_1$ the associated degree.
Then, for any $n \in \{n_1,\ldots,n_k\}$, we have:
\begin{enumerate}
\item
Let $n = p_1^{\nu_1} \cdots p_r^{\nu_r}$ with
pairwise distinct prime numbers, fix
$1 \le s \le r$ and $1 \le \mu \le \nu_s$.
Then $w$ admits an associated degree matrix
\[
\qquad
Q(s,\mu)
\ = \ 
\left[  
\begin{array}{ccc}
w_1 & \ldots & w_{1+d+c}
\\
\zeta_1 & \ldots & \zeta_{1+d+c}  
\end{array}
\right],
\quad
\zeta_i \in \ZZ/p_s^{\mu} \ZZ.
\]
\item
For every prime divisor $p$ of $n$, there are
$0 \le i < j \le 1+d+c$ such that $p \mid l_i$ and
$p \mid l_j$ hold.
\item
If we have $\gcd(w_i,n) = 1$ for some $1 \le i \le 1+d+c$ and
$1 \le j \le k$, then the order $n$ divides the degree $\mu$.
\end{enumerate}
\end{proposition}

\begin{proof}
The fact that $\Gamma$ has at most $d+c-1$ cyclic
factors follows from the fact that any $d+c$ columns of
the degree matrix generate $\ZZ \times \Gamma$.
Statement~(i) is clear by the Chinese remainder
theorem and Proposition~\ref{prop:downgrading}.
For~(ii) note that due to the almost freeness
property of $w$, there are distinct $i,j$ with
$p \nmid w_i, w_j$, and the homogeneity condition
implies $p \mid l_i, l_j$.

We prove~(iii). 
By Proposition~\ref{prop:downgrading}, we may assume
$k=1$.
Adding a suitable multiple of the first row of
$Q$ to the second one,
we achieve $\eta_i=0$.
Then the homogeneity condition $l_i \eta_i = l_j \eta_j$
means $l_j \eta_j = 0$ for all $j$.
Using Lemma~\ref{lem:mullcm}, we conclude that $\mu$
annihilates all entries of the second row of $Q$.
As the latter ones generate $\ZZ/n\ZZ$, we obtain
that $n$ divides $\mu$.
\end{proof}

As a consequence of Proposition~\ref{bem:bounddtors},
we can formulate an algorithm that provides for
any given Gorenstein Fano weight-degree constellation
a redundance-free list of the associated degree
matrices.

\begin{algorithm}[Degree matrices for a given weight-degree constellation]
\label{algo:torsion}
\emph{Input:} a Gorenstein Fano weight-degree constellation $(w,\mu)$
of type $(d,c)$.
\begin{itemize}
\item
Determine the set $P$ of prime numbers that divide one of
$\mu_1, \ldots, \mu_c$. 
\item
For $p \in P$ set $j:=1$ and iterate
\begin{itemize}
\item
if $w$ admits a Gorenstein Fano degree matrix in $\ZZ \times \ZZ/p^j\ZZ$,
then set $j := j+1$,  
\item
if $w$ admits no Gorenstein Fano degree matrix in $\ZZ \times \ZZ/p^j\ZZ$,
then set $\nu_p := j-1$ and stop the iteration.
\end{itemize}
\item
Determine the set $M(w,\mu)$ of all Gorenstein
Fano degree matrices $Q$ associated with $(w,\mu)$
in  
\[
\Gamma \ = \ \ZZ/n_1\ZZ \times \ldots \times \ZZ/n_k\ZZ,
\qquad n_k \mid n_{k-1}, \ldots , n_2 \mid n_1,
\]
where $0 \le k \le d+c-1$ and each of the $n_1$ is a product
of powers $p^j$ with $p \in P$ and $1 \le j \le \nu_p$.
\item
Determine a subset $\mathcal{Q}(w,\mu) \subseteq M(w,\mu)$
such that every matrix from $M(w,\mu)$ is isomorphic to
precisely one element of $\mathcal{Q}(w,\mu)$. 
\end{itemize}
\emph{Output: the set $\mathcal{Q}(w,\mu)$.}
Then $\mathcal{Q}(w,\mu)$ represents in a redundance-free 
manner all Gorenstein Fano degree matrices for $(w,\mu)$.
\end{algorithm}


\begin{remark}
\label{rem:torseff}
If we are in the situation of Proposition~\ref{bem:bounddtors}~(iii),
then we can directly read off the possible $p \in P$ and the
associated $\nu_p$ from the degree vector $\mu$ and thus may skip
the second step in Algorithm~\ref{algo:torsion}.
\end{remark}

%% file: weight-vectors.tex
\section{The weight vectors}

The main results of this section,
Theorems~\ref{thm:onerelwv}, ~\ref{thm:tworelwv}
and~\ref{thm:threerelwv},
present the possible true, Gorenstein, Fano
weight-degree constellations of the types $(3,1)$,
$(3,2)$ and $(3,3)$.


\begin{theorem}
\label{thm:onerelwv}
The true Gorenstein Fano weight-degree constellations
of type $(3,1)$ are precisely the following ones:
\[
\begin{array}{cccc}
(1, 1, 1, 1, 1; \, 2),      &
(1, 1, 1, 1, 1; \, 3),      &
(1, 1, 1, 1, 1; \, 4),      &
(1, 1, 1, 1, 2; \, 4),      \\[3pt]
(1, 1, 1, 1, 3; \, 6),      &
(1, 1, 1, 2, 3; \, 6),      & 
(1, 1, 1, 3, 3; \, 6),      &
(1, 1, 2, 2, 2; \, 4),      \\[3pt]
(1, 1, 2, 2, 2; \, 6),      &
(1, 1, 2, 2, 4; \, 8),      &
(1, 1, 2, 4, 4; \, 8),      &
(1, 1, 2, 4, 6; \, 12),     \\[3pt]
(1, 1, 4, 4, 6; \, 12),     &
(1, 1, 4, 6, 6; \, 12),     &
(1, 2, 2, 2, 3; \, 6),      &
(1, 2, 2, 2, 5; \, 10),     \\[3pt]
(1, 2, 3, 3, 3; \, 6),      &
(1, 2, 3, 6, 6; \, 12),     &
(1, 2, 6, 6, 9; \, 18),     &
(1, 3, 4, 4, 4; \, 12),     \\[3pt]
(1, 3, 8, 12, 12; \, 24),   &
(1, 4, 5, 10, 10; \, 20),   &
(2, 2, 2, 3, 3; \, 6),      &
(2, 3, 3, 4, 6; \, 12).     
\end{array}  
\]
\end{theorem}

\begin{lemma}
\label{lem:3wiequal}
Consider a Gorenstein weight vector
$(w_1, \ldots, w_5)$ of type~$(3,1)$,
two distinct $i_1,i_2$ and the complementary
indices $j_1,j_2,j_3$. Then
\[
\gcd(w_{j_1},w_{j_2},w_{j_3})  \mid (w_{i_1}'+w_{i_2}'),
\qquad
w_{i_k}' := w_{i_k}/\gcd(w_{i_1},w_{i_2}).
\]
In particular, if we have $w_{i_1} = w_{i_2}$, then 
$\gcd(w_{j_1},w_{j_2},w_{j_3})$ equals either one or two,
and in the latter case $w_{i_1} = w_{i_2}$ is odd.
\end{lemma}

\begin{proof}
By definition, $w_0 := \gcd(w_{i_1},w_{i_2})$
divides $w_{j_1}+w_{j_2}+w_{j_3}$.
Also by definition, we have
\[
\gcd(w_{j_1},w_{j_2},w_0)
\ = \
\gcd(w_{j_1},w_{j_3},w_{i_1},w_{i_2})
\ = \
1.
\]
Moreover, $\gcd(w_{j_1},w_{j_2})$ divides
$w_{j_3} + w_{i_1} + w_{i_2} = w_{j_3} + (w_{i_1}'+w_{i_2}')w_0$.
As just observed, the numbers $w_{j_1},w_{j_2},w_0$
are coprime and we can conclude
\[
\gcd(w_{j_1},w_{j_2},w_{j_3}) \mid (w_{i_1}'+w_{i_2}'). 
\]
\end{proof}

\begin{proposition}
\label{prop:l3l4l5known}
Consider a weight vector $w$ of type $(3,1)$,
an exponent vector $l$ for $w$ and let $\mu$ be the
associated degree. Set
\[
\lambda \ := \ \lcm(l_3,l_4,l_5),
\qquad
\lambda_j \ := \ \frac{\lambda}{l_j},
\quad
j = 3,4,5.
\]
Then $\lambda \mid \mu$ and $\lambda_j \mid w_j$ for $j=3,4,5$.
Moreover, $w_i \mid \lambda$ for $i=1,2$.
Finally, setting $\omega := \mu/\lambda$, we
arrive at 
\[
w =  (w_1, w_2, \lambda_3\omega, \lambda_4\omega, \lambda_5\omega),
\qquad
\mu  = \lambda \omega,
\qquad
w_i \mid \lambda, \ \omega \mid l_i, \quad i = 1,2.
\]
If the weight vector $w$ is Gorenstein, then
$\omega$ divides $(w_1+w_2)/\gcd(w_1,w_2)$,
and $w$ as well as $l$ are bounded in terms of
$l_3,l_4,l_5$.
\end{proposition}

\begin{proof}
Recall that $\mu = l_iw_i$ holds for $i=1,\ldots,5$.
Thus each $l_i$ and each $w_i$ divides~$\mu$.
In particular, $\lambda$ divides $\mu$.
For $j=3,4,5$, we further obtain
\[
\mu = l_jw_j = \frac{\lambda}{\lambda_j} w_j
\ \Rightarrow \ 
\frac{\mu}{\lambda} \lambda_j = w_j
\ \Rightarrow \
\ \lambda _j\mid w_j.
\]
Thus, with $\omega = \mu / \lambda$, we see that the
weight vector $w$ looks as in the assertion.
By the almost freeness property, we must have 
\[
\gcd(\omega,w_1) \ = \ \gcd(\omega,w_2) \ = \ 1.
\]
Consequently, for $i=1,2$, we can use
$l_i w_i = \mu = \lambda \omega$
to conclude $w_i \mid \lambda$ and $\omega \mid l_i$.
The statement on the Gorenstein case follows from
Lemma~\ref{lem:3wiequal}.
\end{proof}

\begin{remark}
\label{rem:fano-3-1}
Consider a weight vector $w$ of type $(3,1)$
with exponent vector $l$ and the associated
degree $\mu = l_1w_1$. Then $(w; \mu)$ is
a Fano weight-degree constellation if and
only if 
\[
\frac{1}{l_1} + \ldots + \frac{1}{l_5} \ > \ 1.
\]
\end{remark}

\begin{lemma}
\label{lem:completetriples}
The triples $(l_3,l_4,l_5)$ with $l_3 \ge l_4 \ge l_5 \ge 2$
admitting $l_1 \ge l_2 \ge l_3$ such that $1/l_1+ \ldots + 1/l_5 > 1$
are precisely the following
\[
\begin{array}{lclcl}  
(l_3,2,2), \ l_3 =  2, \ldots, \infty,
& &
(l_3,3,2), \ l_3 = 3, \ldots, 17,
& &
(l_3,4,2), \ l_3 = 4, \ldots, 11,
\\[3pt]
(l_3,5,2), \ l_3 = 5, \ldots, 9,
& &
(l_3,6,2), \ l_3 = 6, 7, 8,
& &
(l_3,7,2), \ l_3 = 7, 8,
\\[3pt]
(l_3,3,3), \ l_3 = 3, \ldots, 8,
& &
(l_3,4,3), \ l_3 = 4, \ldots, 7,
& &
(l_3,5,3), \ l_3 = 5, 6,
\\[3pt]
(l_3,4,4), \ l_3 = 4, 5,
& &
(5,5,4).
& &
\end{array}       
\]  
\end{lemma}

\begin{proof}
Obviously, all the triples $(l_3,2,2)$ with $l_3 \ge 2$
allow the wanted extension. Moreover, observe that
\[
\frac{1}{2} + \frac{4}{l_4} \ > \ 1,
\quad
\frac{1}{2} + \frac{1}{l_4} + \frac{3}{l_3} \ > \ 1,
\quad
\frac{1}{2} + \frac{1}{l_4} + \frac{1}{l_3} + \frac{2}{l_2}  \ > \ 1,
\quad
\frac{1}{2} + \frac{1}{l_4} + \frac{1}{l_3} + \frac{1}{l_2}   + \frac{1}{l_1}  \ > \ 1
\]
successively bound $l_4, \ldots, l_1$ in the case
$l_4 \ge 3$. The assertion is then obtained by explicit
computation.
\end{proof}

\begin{proposition}
\label{prop:345expy22}
The Gorenstein weight vectors $w$ of type
$(3,1)$ with associated exponent vector
of the form $l=(l_1,l_2,y,2,2)$ are
precisely 
\[
\begin{array}{ccccc}
(1,1,1,1,1), & (1,1,1,3,3), & (1,1,2,2,2), & (1,1,2,4,4), & (1,1,4,6,6), 
\\[3pt]
(1,2,3,3,3),  & (1,2,3,6,6), & (1,4,5,10,10), &  (1,3,8,12,12), & (2,2,3,3,3). 
\end{array}
\]
\end{proposition}

\begin{proof}
We make permanent use of Proposition~\ref{prop:l3l4l5known}.
In the notation introduced there, we have
\[
\lambda
\ = \
\begin{cases}
y, & y \text{ even},
\\  
2y, & y \text{ odd},
\end{cases}  
\qquad\qquad
w
\ = \
\begin{cases}
\left(w_1, w_2, \omega, \frac{y}{2}\omega, \frac{y}{2}\omega\right), & y \text{ even},
\\  
\left(w_1, w_2, 2\omega, y\omega, y\omega\right), & y \text{ odd}.
\end{cases}  
\]

We treat the case that $y$ is even. Then $y = \lambda = k_1w_1 = k_2w_2$ with
integers $k_1 \ge k_2$.
Lemma~\ref{lem:3wiequal} tells us $w_1 + w_2 = \alpha \omega$ for some $\alpha \in \ZZ_{\ge 0}$
and thus
\[
\omega
\ = \
\frac{1}{\alpha} \left( w_1 + w_2 \right) 
\ = \
\frac{1}{\alpha} \left( \frac{y}{k_1} + \frac{y}{k_2}  \right) 
\ = \
\frac{1}{\alpha} \cdot \left(\frac{1}{k_1} + \frac{1}{k_2}\right) \cdot y.
\]
This leaves us with $\alpha=1,2$, because using $y/k_2 = w_2 \le w_3 = \omega$,
we can estimate
\[
\frac{1}{k_2}
\ \le \
\frac{1}{\alpha} \cdot \left(\frac{1}{k_1} + \frac{1}{k_2}\right).
\]

Let $\alpha=1$. Then $\omega=w_1+w_2$ and, consequently, the
weight vector $w$ can be written as
\[
w \ = \ \left(w_1, \ w_2, \ w_1+w_2, \ \frac{y}{2}(w_1+w_2), \ \frac{y}{2}(w_1+w_2) \right).
\]
By the Gorenstein property, the sum $2(w_1+w_2)$ of the first three entries
is a multiple of the last entry. This leads to 
\[
2 \ = \ \frac{y}{2} \cdot \beta
\]  
for some $\beta \in \ZZ_{>0}$.  
We conclude $y = 1,2,4$, where $y=1$ is ruled out by $w_3 \le w_4$.
For $y=2$ and $y=4$, the Gorenstein weight vectors are
\[
(1,1,2,2,2), \ (1,2,3,3,3),
\qquad
(1,1,2,4,4), \ (1,2,3,6,6), \ (1,4,5,10,10).
\]

Now, let $\alpha=2$. Then $k_1=k_2$, hence $w_1=w_2$
and $\omega=w_1$. 
Consequently, the weight vector looks as follows:
\[
w \ = \ \left(w_1,w_1,w_1,\frac{y}{2}w_1,\frac{y}{2}w_1\right).
\]
By the Gorenstein property, $3w_1$ is a multiple of
the last entry of $w$. Thus, $y=2,6$. The possible
Gorenstein weight vectors in this situation are
precisely
\[
(1,1,1,1,1), \qquad (1,1,1,3,3).
\]

We treat the case that $y$ is odd. Then $2y = \lambda = k_1w_1 = k_2w_2$ with
integers $k_1 \ge k_2$.
Lemma~\ref{lem:3wiequal} tells us $w_1 + w_2 = \alpha \omega$ for some $\alpha \in \ZZ_{\ge 0}$
and thus
\[
\omega
\ = \
\frac{1}{\alpha} \left( w_1 + w_2 \right) 
\ = \
\frac{2}{\alpha} \left( \frac{y}{k_1} + \frac{y}{k_2}  \right) 
\ = \
\frac{2}{\alpha} \cdot \left(\frac{1}{k_1} + \frac{1}{k_2}\right) \cdot y.
\]
Thus, using $2y/k_2 = w_2 \le w_3 = 2\omega$, we can estimate
\[
\frac{1}{k_2}
\ \le \
\frac{2}{\alpha} \cdot \left(\frac{1}{k_1} + \frac{1}{k_2}\right).
\]
This allows $\alpha=1,2,3,4$, where $\omega=(w_1+w_2)/\alpha$
must be an integer in each case.
Accordingly, the weight vector $w$ can be written as
\[
\begin{array}{lcl}
w & = & \left(w_1, \ w_2, \ 2(w_1+w_2), \ y(w_1+w_2), \ y(w_1+w_2) \right),
\\[3pt]
w & = & \left(w_1, \ w_2, \ w_1+w_2, \ y \frac{w_1+w_2}{2}, \ y \frac{w_1+w_2}{2} \right),
\\[3pt]
w & = & \left(w_1, \ w_2, \ 2 \frac{w_1+w_2}{3}, \ y \frac{w_1+w_2}{3}, \ y \frac{w_1+w_2}{3} \right),
\\[3pt]
w & = & \left(w_1, \ w_1, \ w_1, \ y \frac{w_1}{2}, \ y \frac{w_1}{2} \right).
\end{array}
\]

In the first setting, the Gorenstein property implies that
$3(w_1+w_2)$ is a multiple of the last entry.
This leaves us with $y=3$ and $w_1,w_2 \mid 6$, and
the resulting $w$ are
\[
(1,1,4,6,6), \qquad (1,3,8,12,12).
\]

In the second setting, $2(w_1+w_2)$ is a multiple of the last
entry. Thus $y \mid 4$. As~$y$ is odd, we end up with $y=1$,
which hurts the ascending order of the $w_i$.
The third setting forces $5(w_1+w_2)$ to be a multiple of the last
entry. This leads to $y = 5$ and $w_1,w_2 \mid 10$.
Checking all possibilities, we see that no Gorenstein weight
vector pops up here.

Finally, in the last setting, that means $\alpha=4$, we obtain
that $3w_1$ is a multiple of the last entry. Thus, $y \mid 6$.
As $y$ is odd and the weights are in ascending order, we
arrive at $y=3$ and $w_1 \mid 6$. This forces $w_1=2$ and
the resulting Gorenstein weight vector is
\[
(2,2,2,3,3).
\]
\end{proof}

\begin{proposition}
\label{prop:345expy32}
The Gorenstein weight vectors $w$ of type
$(3,1)$ allowing a Fano exponent vector
$l=(l_1,l_2,l_3,l_4,l_5)$ with $l_4 \ge 3$
are precisely 
\[
\begin{array}{lllll}
(1, 1, 1, 1, 1), & (1, 1, 1, 1, 2), & (1, 1, 1, 1, 3), & (1, 1, 1, 2, 3), & (1, 1, 2, 2, 2),
\\[3pt]
(1, 1, 2, 2, 4), & (1, 1, 2, 4, 6), & (1, 1, 4, 4, 6), & (1, 2, 2, 2, 3), & (1, 2, 2, 2, 5),
\\[3pt]
(1, 2, 6, 6, 9), & (1, 3, 4, 4, 4), & (2, 3, 3, 4, 6). & &
\end{array}
\]
\end{proposition}

\begin{proof}
Remark~\ref{rem:fano-3-1} yields $1/l_1 + \ldots + 1/l_5 > 1$.
Thus, we can apply Lemma~\ref{lem:completetriples},
which, due to $l_4 \ge 3$, provides us with 
explicit bounds for the entries $l_3,l_4,l_5$ of the
exponent vector $l$.
Now Proposition~\ref{prop:l3l4l5known} shows how to
bound the entries of the weight vector $w$ and
we succeed computationally.
\end{proof}

\begin{proof}[Proof of Theorem~\ref{thm:onerelwv}]
Propositions~\ref{prop:345expy22} and~\ref{prop:345expy32}
deliver the weight vectors.
The possible degrees are directly determined via the
Fano condition.
\end{proof}

\begin{proposition}
\label{prop:gorw3equal}
Let $w = (w_1, \ldots, w_5)$ be a Gorenstein weight
vector of type~$(3,1)$ admitting pairwise distinct
$i_1,i_2,i_3$ with $w_{i_1} = w_{i_2} = w_{i_3}$. Then
\[
\gcd(w_{j_1},w_{j_2}) \mid 3,
\qquad  
w_{i1}\gcd(w_{j_1},w_{j_2}) \mid (w_{j_1} + w_{j_2}),
\]
where $j_1,j_2$ are the complementary indices. Moreover,
if in addition $w_{j_1} = w_{j_2}$ holds, then the
weight vector $w$ is one of the following:
\[
(1,1,1,1,1), \qquad
(1,1,1,3,3), \qquad
(1,1,2,2,2), \qquad
(2,2,2,3,3).
\]
\end{proposition}

\begin{proof}
The definition of a weight vector shows
$\gcd(w_{i_1},w_{j_1}) = \gcd(w_{i_1},w_{j_2}) = 1$
and $\gcd(w_{j_1},w_{j_2}) \mid 3w_{i_1}$.
Thus, $\gcd(w_{j_1},w_{j_2})$ divides $3$.
Using Lemma~\ref{lem:3wiequal}, we see that
$w_{i1}\gcd(w_{j_1},w_{j_2})$ divides $w_{j_1}+w_{j_2}$.
If $w_{j_1} = w_{j_2}$, then either both equal one
or both equal three and we end up with the weight
vectors listed in the assertion.
\end{proof}


\begin{theorem}
\label{thm:tworelwv}
The true Gorenstein Fano weight-degree constellations
of type $(3,2)$ are precisely the following ones:
\[
(1,1,1,1,1,1; \, 2,2),
\quad
(1,1,1,1,1,1; \, 2,3),
\quad
(1,1,2,2,2,2; \, 4,4),
\]
\[
(1,2,3,3,3,3; \, 6,6),
\qquad\quad
(2,2,2,2,3,3; \, 6,6).
\]
\end{theorem}

\begin{remark}
\label{rem:fano-3-2}
Consider a weight vector $w$ of type $(3,2)$
with exponent vectors $l,l'$, where
$l_1 \le l_1'$, and the associated degrees
$\mu, \mu'$.
Then $(w; \mu, \mu')$ is a Fano weight-degree
constellation only if 
\[
\frac{1}{2l_1} + \ldots + \frac{1}{2l_6} \ > \ 1.
\]
\end{remark}

\begin{lemma}
\label{lem:completetriples2}
The quintuples $(l_2,\ldots, l_6)$ with
$l_2 \ge \ldots  \ge l_6 \ge 2$
admitting $l_1 \ge l_2$ such that $1/2l_1+ \ldots + 1/2l_6 > 1$
are precisely the following
\[
\begin{array}{lcl}  
(l_2,2,2,2,2), \ l_2 =  2, \ldots, \infty,
& \qquad &
(4,4,2,2,2), 
\\[3pt]
(l_2,3,2,2,2), \ l_3 = 3, \ldots, 6,
& \qquad &
(3,3,3,2,2).
\end{array}       
\]  
\end{lemma}

\begin{proof}
The arguments from the proof of Lemma~\ref{lem:completetriples}
can be directly adapted.  
\end{proof}

\begin{proposition}
\label{prop:expvectll2222}
Consider a Gorenstein weight vector $w$
of type $(3,2)$ admitting $l = (l_1,l_2,2,2,2,2)$
as exponent vector. Then $l$ is one of the following:
\[
(2,2,2,2,2,2),
\qquad\qquad
(4,4,2,2,2,2),  
\qquad\qquad
(6,3,2,2,2,2).
\]
According to these cases, the weight vector $w$ equals
$(1,1,1,1,1,1)$, $(1,1,2,2,2,2)$, or $(1,2,3,3,3,3)$.
\end{proposition}

\begin{proof}
Only for the statement on the possible choices
for $l$ there is something to show.
The weight vector is of the form
$w = (w_1,w_2,\omega,\omega,\omega,\omega)$.
Due to the properties of a weight vector,
$\omega$ is coprime to each of $w_1, w_2$.
Homogeneity yields $l_iw_i=2\omega$,
thus $w_i \mid 2$ and we are left with
\[
(w_1,w_2) = (1,1), \qquad
(w_1,w_2) = (1,2), \qquad
(w_1,w_2) = (2,2).
\]

\noindent
\emph{Case $(w_1,w_2) = (1,1)$.}
Then $w =(1,1,\omega,\omega,\omega,\omega)$
and $\omega$ divides $1+1+\omega$
due to the Gorenstein property. 
Consequently, $\omega \mid 2$ and thus $\omega=1,2$,
which gives the first and the second possibility
for $l$ from the assertion.

\medskip

\noindent
\emph{Case $(w_1,w_2) = (1,2)$.}
Then $w =(1,2,\omega,\omega,\omega,\omega)$
and $\omega \mid (1+2+\omega)$
by the Gorenstein property.
Thus, $\omega \mid 3$.
Ascending order of $w$ excludes $\omega=1$,
hence $\omega=3$. This gives the third
possibility for $l$ from
the assertion.

\medskip

\noindent
\emph{Case $(w_1,w_2) = (2,2)$.}
Then $w =(2,2,\omega,\omega,\omega,\omega)$
and $\omega$ divides $2+2+\omega$,
due to the Gorenstein property.
Thus, $\omega \mid 4$.
The properties of a weight vector force
$\omega=1$, violating ascending order of $w$.
\end{proof}

\begin{proposition}
Apart from those from Proposition~\ref{prop:expvectll2222},
there is only one true Gorenstein Fano weight-degree constellation
of type $(3,2)$, namely $(2,2,2,2,3,3; \, 6,6)$.
\end{proposition}

\begin{proof}
According to Lemma~\ref{lem:completetriples2}, the task is to
exclude except $(3,3,3,3,2,2)$ all exponent vectors having as
their tuple of last five entries one of the following:
\[
(l_2,3,2,2,2), \ l_2 = 3, \ldots, 6,
\qquad
(3,3,3,2,2),
\qquad
(4,4,2,2,2).
\]
Lemma~\ref{lem:mullcm} tells us that any possible first
entry $l_1$ must divide the least common multiple of the
last five. Thus the choice of possible exponent vectors is
bounded and we obtain the assertion computationally.
\end{proof}


\begin{theorem}
\label{thm:threerelwv}
There is precisely one true Gorenstein Fano weight-degree
constellation of type $(3,3)$, namely $(1,1,1,1,1,1,1; \, 2,2,2)$.
\end{theorem}

\begin{proposition}
\label{prop:3relfanowv}
Let $(\ell_1,\ell_2,\ell_3)$ be the triple of exponent vectors
of a true Fano weight-degree constellation of type $(3,3)$.
Then one of the following holds:
\[
\ell_1 = \ell_2 = \ell_3 = (2,2,2,2,2,2,2),
\qquad\qquad
\ell_1 = \ell_2 = \ell_3 = (3,3,2,2,2,2,2).
\]
\end{proposition}

\begin{proof}
Remark~\ref{rem:altfano} yields $\ell_{1,7}+\ell_{2,7}+\ell_{3,7} \le 6$.
Consequently, $\ell_{1,7}=\ell_{2,7}=\ell_{3,7}=2$.
Using homogeneity, we obtain
\[
\ell_{1,j}w_j \ = \  \ell_{2,j}w_j \ = \ \ell_{3,j}w_j \ = \ 2w_7, \qquad j = 1, \ldots, 6.
\]
That means $\ell_1 = \ell_2 = \ell_3 =: l$ and the exponent vector
$l$ satisfies $l_1 \ge \ldots \ge l_7 = 2$.
The first estimate of Remark~\ref{rem:altfano} tells us
\[
\frac{1}{3l_1} + \ldots + \frac{1}{3l_6} + \frac{1}{6} \ > \ 1.
\]
We directly conclude $l_3 =  \ldots = l_6 = 2$.
Plugging this into the above inequality, we see that
the pair $(l_1,l_2)$ must satisfy 
\[
\frac{1}{l_1} + \frac{1}{l_2} \ > \ \frac{1}{2}.
\]
The possibilities are $(l_1,2)$ with $2 \le l_1$
and $(l_1,3)$ with $3 \le l_1 \le 5$.
In the case $(l_1,2)$, the exponent vector $l$ and the weight
vector $w$ are of the form
\[
l \ = \ (l_1,2,2,2,2,2,2),  
\qquad
w \ = \ (w_1,w_2,w_2,w_2,w_2,w_2,w_2).
\qquad
\]
As any six entries of the weight vector are coprime, we conclude
$w_2=1$. This implies $l_1w_1=2$, forcing $l_1=2$ and $w_1=1$.
If $(l_1,l_2) = (4,3)$, then
\[
l \ = \ (4,3,2,2,2,2,2),  
\qquad
w \ = \ (w_1,w_2,2w_1,2w_1,2w_1,2w_1,2w_1).
\qquad
\]
By homogeneity, $3w_2=4w_1$ must hold. Thus, 
$4 \mid w_2$, which contradicts to coprimeness
of the last six weights.
Now, consider $(l_1,l_2) = (5,3)$, then
\[
l \ = \ (5,3,2,2,2,2,2),  
\qquad
w \ = \ (w_1,w_2,w_3,w_3,w_3,w_3,w_3),
\qquad
\]
where $5w_1=3w_2=2w_3$. This implies
$5 \mid w_2$ and $5 \mid w_3$,
contradicting coprimeness of $w_2$ and $w_3$.
We end up with $(l_1,l_2)$ being 
$(2,2)$ or $(3,3)$.
\end{proof}

\begin{proof}[Proof of Theorem~\ref{thm:threerelwv}]
We have to exclude the triple $(l,l,l)$ of exponent vectors
with $l = (3,3,2,2,2,2,2)$
from Proposition~\ref{prop:3relfanowv}.
In this setting, the weight vector is of the shape
$(w_1,w_1,w_2,w_2,w_2,w_2,w_2)$.
Homogeneity implies $3w_1 = 2w_2$.
We conclude $3 \mid w_2$.
The Gorenstein property~\ref{def:gorenstein}
yields $3 \mid (2w_1+w_2)$ and thus $3 \mid w_1$.
This contradicts to Definition~\ref{def:wvev}.
\end{proof}

%% file: classification-lists.tex
\section{Classification lists}

Here we present our classification lists for
the general Gorenstein Fano non-degenerate
toric complete intersection threefolds of
rank one.
Theorems~\ref{thm:onerelwv}, ~\ref{thm:tworelwv}
and~\ref{thm:threerelwv} provide us with the
possible true Fano weight-degree constellations.
A suitable implementation of Algorithm~\ref{algo:torsion}
provides us with the possible associated degree matrices,
which complements the classification.
Proposition~\ref{prop:degdata2gtci}
verifies each particular case and,
finally, Propositions~\ref{prop:gtciprops1}
and~\ref{prop:h0KX} enable us to compute
the geometric invariants in each case.


\setcounter{tabits}{0}
\begin{classlist}
\label{thm:classif-w11111}
Degree and geometric data of $\QQ$-factorial,
Gorenstein, Fano gtci threefolds of rank one,
type $(3,1)$ and weight vector $(1,1,1,1,1)$:
\begin{center}
\begin{longtable}{ccccc}
{\small ID} & {\small degree data}  & {\small $-\mathcal{K}$} & {\small $-\mathcal{K}^3$} & {\small $h^0(-\mathcal{K})$}
%
\input{latex-matrices-w11111t1.txt}
%
\input{latex-matrices-w11111t3.txt}
\end{longtable}  
\end{center}
\vspace*{-10pt}
Every $\QQ$-factorial, Gorenstein, Fano gtci
threefold of rank one, type $(3,1)$
and weight vector $(1,1,1,1,1)$ is isomorphic
to precisely one of the \thetabits\ listed items.
\end{classlist}

\setcounter{tabits}{0}
\begin{classlist}
\label{thm:classif-w11112}
Degree and geometric data of $\QQ$-factorial,
Gorenstein, Fano gtci threefolds of rank one,
type $(3,1)$ and weight vector $(1,1,1,1,2)$:
\begin{center}
\begin{longtable}{ccccc}
{\small ID} & {\small degree data}  & {\small $-\mathcal{K}$} & {\small $-\mathcal{K}^3$} & {\small $h^0(-\mathcal{K})$}
%
\input{latex-matrices-w11112t1.txt}
%
\input{latex-matrices-w11112t2.txt}
%
\input{latex-matrices-w11112t4.txt}
\end{longtable}  
\end{center}
\vspace*{-10pt}
Every $\QQ$-factorial Gorenstein Fano gtci threefold of
rank one, type $(3,1)$
and weight vector $(1,1,1,1,2)$ is isomorphic to precisely
one of the \thetabits\ listed items.
\end{classlist}

\setcounter{tabits}{0}
\begin{classlist}
\label{thm:classif-w11113}
Degree and geometric data of $\QQ$-factorial,
Gorenstein, Fano gtci threefolds of rank one,
type $(3,1)$ and weight vector $(1,1,1,1,3)$:
\begin{center}
\begin{longtable}{ccccc}
{\small ID} & {\small degree data}  & {\small $-\mathcal{K}$} & {\small $-\mathcal{K}^3$} & {\small $h^0(-\mathcal{K})$}
%
\input{latex-matrices-w11113t1.txt}
\end{longtable}  
\end{center}
\vspace*{-10pt}
Every $\QQ$-factorial Gorenstein Fano gtci threefold
of rank one, type $(3,1)$ 
and weight vector $(1,1,1,1,3)$ is isomorphic to the
above item.
\end{classlist}

\setcounter{tabits}{0}
\begin{classlist}
\label{thm:classif-w11123}
Degree and geometric data of $\QQ$-factorial,
Gorenstein, Fano gtci threefolds of rank one,
type $(3,1)$ and weight vector $(1,1,1,2,3)$:
\begin{center}
\begin{longtable}{ccccc}
{\small ID} & {\small degree data}  & {\small $-\mathcal{K}$} & {\small $-\mathcal{K}^3$} & {\small $h^0(-\mathcal{K})$}
%
\input{latex-matrices-w11123t1.txt}
%
\input{latex-matrices-w11123t2.txt}
%
\input{latex-matrices-w11123t22.txt}
\end{longtable}  
\end{center}
\vspace*{-10pt}
Every $\QQ$-factorial Gorenstein Fano gtci of rank
one, type $(3,1)$ and weight vector $(1,1,1,2,3)$ is isomorphic
to precisely one of the \thetabits\ listed items.
\end{classlist}

\setcounter{tabits}{0}
\begin{classlist}
\label{thm:classif-w11133}
Degree and geometric data of $\QQ$-factorial,
Gorenstein, Fano gtci threefolds of rank one,
type $(3,1)$ and weight vector $(1,1,1,3,3)$:
\begin{center}
\begin{longtable}{ccccc}
{\small ID} & {\small degree data}  & {\small $-\mathcal{K}$} & {\small $-\mathcal{K}^3$} & {\small $h^0(-\mathcal{K})$}
%
\input{latex-matrices-w11133t1.txt}
%
\input{latex-matrices-w11133t3.txt}
\end{longtable}  
\end{center}
\vspace*{-10pt}
Every $\QQ$-factorial Gorenstein Fano gtci threefold of
rank one, type $(3,1)$ and weight
vector $(1,1,1,3,3)$ is isomorphic to precisely one of the
\thetabits\ listed items.
\end{classlist}

\setcounter{tabits}{0}
\begin{classlist}
\label{thm:classif-w11222}
Degree and geometric data of $\QQ$-factorial,
Gorenstein, Fano gtci threefolds of rank one,
type $(3,1)$ and weight vector $(1,1,2,2,2)$:
\begin{center}
\begin{longtable}{ccccc}
{\small ID} & {\small degree data}  & {\small $-\mathcal{K}$} & {\small $-\mathcal{K}^3$} & {\small $h^0(-\mathcal{K})$}
%
\input{latex-matrices-w11222t1.txt}
%
\input{latex-matrices-w11222t2.txt}
\end{longtable}  
\end{center}
\vspace*{-10pt}
Every $\QQ$-factorial Gorenstein Fano gtci threefold of
rank one, type $(3,1)$ and weight
vector $(1,1,2,2,2)$ is isomorphic to precisely one of the
\thetabits\ listed items.
\end{classlist}

\setcounter{tabits}{0}
\begin{classlist}
\label{thm:classif-w11224}
Degree and geometric data of $\QQ$-factorial,
Gorenstein, Fano gtci threefolds of rank one,
type $(3,1)$ and weight vector $(1,1,2,2,4)$:
\begin{center}
\begin{longtable}{ccccc}
{\small ID} & {\small degree data}  & {\small $-\mathcal{K}$} & {\small $-\mathcal{K}^3$} & {\small $h^0(-\mathcal{K})$}
%
\input{latex-matrices-w11224t1.txt}
%
\input{latex-matrices-w11224t2.txt}
\end{longtable}  
\end{center}
\vspace*{-10pt}
Every $\QQ$-factorial Gorenstein Fano gtci threefold of
rank one, type $(3,1)$ and weight
vector $(1,1,2,2,4)$ is isomorphic to precisely one of the
\thetabits\ listed items.
\end{classlist}

\setcounter{tabits}{0}
\begin{classlist}
\label{thm:classif-w11244}
Degree and geometric data of $\QQ$-factorial,
Gorenstein, Fano gtci threefolds of rank one,
type $(3,1)$ and weight vector $(1,1,2,4,4)$:
\begin{center}
\begin{longtable}{ccccc}
{\small ID} & {\small degree data}  & {\small $-\mathcal{K}$} & {\small $-\mathcal{K}^3$} & {\small $h^0(-\mathcal{K})$}
%
\input{latex-matrices-w11244t1.txt}
%
\input{latex-matrices-w11244t2.txt}
\end{longtable}  
\end{center}
\vspace*{-10pt}
Every $\QQ$-factorial Gorenstein Fano gtci
threefold of rank one, $(3,1)$ and weight
vector $(1,1,2,4,4)$ is isomorphic to precisely one of the
\thetabits\ listed items.
\end{classlist}

\setcounter{tabits}{0}
\begin{classlist}
\label{thm:classif-w11246}
Degree and geometric data of $\QQ$-factorial,
Gorenstein, Fano gtci threefolds of rank one,
type $(3,1)$ and weight vector $(1,1,2,4,6)$:
\begin{center}
\begin{longtable}{ccccc}
{\small ID} & {\small degree data}  & {\small $-\mathcal{K}$} & {\small $-\mathcal{K}^3$} & {\small $h^0(-\mathcal{K})$}
%
\input{latex-matrices-w11246t1.txt}
\end{longtable}  
\end{center}
\vspace*{-10pt}
Every $\QQ$-factorial Gorenstein Fano gtci threefold of
rank one, type $(3,1)$ and weight
vector $(1,1,2,4,6)$ is isomorphic to the above
item.
\end{classlist}

\setcounter{tabits}{0}
\begin{classlist}
\label{thm:classif-w11446}
Degree and geometric data of $\QQ$-factorial,
Gorenstein, Fano gtci threefolds of rank one,
type $(3,1)$ and weight vector $(1,1,4,4,6)$:
\begin{center}
\begin{longtable}{ccccc}
{\small ID} & {\small degree data}  & {\small $-\mathcal{K}$} & {\small $-\mathcal{K}^3$} & {\small $h^0(-\mathcal{K})$}
%
\input{latex-matrices-w11446t1.txt}
%
\input{latex-matrices-w11446t2.txt}
\end{longtable}  
\end{center}
\vspace*{-10pt}
Every $\QQ$-factorial Gorenstein Fano gtci 
threefold of rank one, type $(3,1)$ and weight
vector $(1,1,4,4,6)$ is isomorphic to precisely one of the
\thetabits\ listed items.
\end{classlist}

\setcounter{tabits}{0}
\begin{classlist}
\label{thm:classif-w11466}
Degree and geometric data of $\QQ$-factorial,
Gorenstein, Fano gtci threefolds of rank one,
type $(3,1)$ and weight vector $(1,1,4,6,6)$:
\begin{center}
\begin{longtable}{ccccc}
{\small ID} & {\small degree data}  & {\small $-\mathcal{K}$} & {\small $-\mathcal{K}^3$} & {\small $h^0(-\mathcal{K})$}
%
\input{latex-matrices-w11466t1.txt}
%
\input{latex-matrices-w11466t3.txt}
\end{longtable}  
\end{center}
\vspace*{-10pt}
Every $\QQ$-factorial Gorenstein Fano gtci
threefold of rank one, type $(3,1)$ and weight
vector $(1,1,4,6,6)$ is isomorphic to precisely one of the
\thetabits\ listed items.
\end{classlist}

\setcounter{tabits}{0}
\begin{classlist}
\label{thm:classif-w12223}
Degree and geometric data of $\QQ$-factorial,
Gorenstein, Fano gtci threefolds of rank one,
type $(3,1)$ and weight vector $(1,2,2,2,3)$:
\begin{center}
\begin{longtable}{ccccc}
{\small ID} & {\small degree data}  & {\small $-\mathcal{K}$} & {\small $-\mathcal{K}^3$} & {\small $h^0(-\mathcal{K})$}
%
\input{latex-matrices-w12223t1.txt}
\end{longtable}  
\end{center}
\vspace*{-10pt}
Every $\QQ$-factorial Gorenstein Fano gtci
threefold of rank one, type $(3,1)$ and weight
vector $(1,2,2,2,3)$ is isomorphic to the above item.
\end{classlist}

\setcounter{tabits}{0}
\begin{classlist}
\label{thm:classif-w12225}
Degree and geometric data of $\QQ$-factorial,
Gorenstein, Fano gtci threefolds of rank one,
type $(3,1)$ and weight vector $(1,2,2,2,5)$:
\begin{center}
\begin{longtable}{ccccc}
{\small ID} & {\small degree data}  & {\small $-\mathcal{K}$} & {\small $-\mathcal{K}^3$} & {\small $h^0(-\mathcal{K})$}
%
\input{latex-matrices-w12225t1.txt}
\end{longtable}  
\end{center}
\vspace*{-10pt}
Every $\QQ$-factorial Gorenstein Fano gtci
threefold of rank one, type $(3,1)$ and weight
vector $(1,2,2,2,5)$ is isomorphic to the above item.
\end{classlist}

\setcounter{tabits}{0}
\begin{classlist}
\label{thm:classif-w12333}
Degree and geometric data of $\QQ$-factorial,
Gorenstein, Fano gtci threefolds of rank one,
type $(3,1)$ and weight vector $(1,2,3,3,3)$:
\begin{center}
\begin{longtable}{ccccc}
{\small ID} & {\small degree data}  & {\small $-\mathcal{K}$} & {\small $-\mathcal{K}^3$} & {\small $h^0(-\mathcal{K})$}
%
\input{latex-matrices-w12333t1.txt}
%
\input{latex-matrices-w12333t2.txt}
%
\input{latex-matrices-w12333t22.txt}
\end{longtable}  
\end{center}
\vspace*{-10pt}
Every $\QQ$-factorial Gorenstein Fano gtci
threefold of rank one, type $(3,1)$ and weight
vector $(1,2,3,3,3)$ is isomorphic to precisely one of the
\thetabits\ listed items.
\end{classlist}

\setcounter{tabits}{0}
\begin{classlist}
\label{thm:classif-w12366}
Degree and geometric data of $\QQ$-factorial,
Gorenstein, Fano gtci threefolds of rank one,
type $(3,1)$ and weight vector $(1,2,3,6,6)$:
\begin{center}
\begin{longtable}{ccccc}
{\small ID} & {\small degree data}  & {\small $-\mathcal{K}$} & {\small $-\mathcal{K}^3$} & {\small $h^0(-\mathcal{K})$}
%
\input{latex-matrices-w12366t1.txt}
%
\input{latex-matrices-w12366t2.txt}
\end{longtable}  
\end{center}
\vspace*{-10pt}
Every $\QQ$-factorial Gorenstein Fano gtci
threefold of rank one, type $(3,1)$ and weight
vector $(1,2,3,6,6)$ is isomorphic to precisely one of the
\thetabits\ listed items.
\end{classlist}

\setcounter{tabits}{0}
\begin{classlist}
\label{thm:classif-w12669}
Degree and geometric data of $\QQ$-factorial,
Gorenstein, Fano gtci threefolds of rank one,
type $(3,1)$ and weight vector $(1,2,6,6,9)$:
\begin{center}
\begin{longtable}{ccccc}
{\small ID} & {\small degree data}  & {\small $-\mathcal{K}$} & {\small $-\mathcal{K}^3$} & {\small $h^0(-\mathcal{K})$}
%
\input{latex-matrices-w12669t1.txt}
\end{longtable}  
\end{center}
\vspace*{-10pt}
Every $\QQ$-factorial Gorenstein Fano gtci
threefold of rank one, type $(3,1)$ and weight
vector $(1,2,6,6,9)$ is isomorphic to the above item.
\end{classlist}

\setcounter{tabits}{0}
\begin{classlist}
\label{thm:classif-w13444}
Degree and geometric data of $\QQ$-factorial,
Gorenstein, Fano gtci threefolds of rank one,
type $(3,1)$ and weight vector $(1,3,4,4,4)$:
\begin{center}
\begin{longtable}{ccccc}
{\small ID} & {\small degree data}  & {\small $-\mathcal{K}$} & {\small $-\mathcal{K}^3$} & {\small $h^0(-\mathcal{K})$}
%
\input{latex-matrices-w13444t1.txt}
\end{longtable}  
\end{center}
\vspace*{-10pt}
Every $\QQ$-factorial Gorenstein Fano gtci
threefold of rank one, type $(3,1)$ and weight
vector $(1,3,4,4,4)$ is isomorphic to the above item.
\end{classlist}

\setcounter{tabits}{0}
\begin{classlist}
\label{thm:classif-w138CC}
Degree and geometric data of $\QQ$-factorial,
Gorenstein, Fano gtci threefolds of rank one,
type $(3,1)$ and weight vector $(1,3,8,12,12)$:
\begin{center}
\begin{longtable}{ccccc}
{\small ID} & {\small degree data}  & {\small $-\mathcal{K}$} & {\small $-\mathcal{K}^3$} & {\small $h^0(-\mathcal{K})$}
%
\input{latex-matrices-w138CCt1.txt}
\end{longtable}  
\end{center}
\vspace*{-10pt}
Every $\QQ$-factorial Gorenstein Fano gtci
threefold of rank one, type $(3,1)$ and weight
vector $(1,3,8,12,12)$ is isomorphic to the above item.
\end{classlist}

\setcounter{tabits}{0}
\begin{classlist}
\label{thm:classif-w145AA}
Degree and geometric data of $\QQ$-factorial,
Gorenstein, Fano gtci threefolds of rank one,
type $(3,1)$ and weight vector $(1,4,5,10,10)$:
\begin{center}
\begin{longtable}{ccccc}
{\small ID} & {\small degree data}  & {\small $-\mathcal{K}$} & {\small $-\mathcal{K}^3$} & {\small $h^0(-\mathcal{K})$}
%
\input{latex-matrices-w145AAt1.txt}
\end{longtable}  
\end{center}
\vspace*{-10pt}
Every $\QQ$-factorial Gorenstein Fano gtci
threefold of rank one, type $(3,1)$ and weight
vector $(1,4,5,10,10)$ is isomorphic the above item.
\end{classlist}

\setcounter{tabits}{0}
\begin{classlist}
\label{thm:classif-w22233}
Degree and geometric data of $\QQ$-factorial,
Gorenstein, Fano gtci threefolds of rank one,
type $(3,1)$ and weight vector $(2,2,2,3,3)$:
\begin{center}
\begin{longtable}{ccccc}
{\small ID} & {\small degree data}  & {\small $-\mathcal{K}$} & {\small $-\mathcal{K}^3$} & {\small $h^0(-\mathcal{K})$}
%
\input{latex-matrices-w22233t1.txt}
\end{longtable}  
\end{center}
\vspace*{-10pt}
Every $\QQ$-factorial Gorenstein Fano gtci
threefold of rank one, type $(3,1)$ and weight
vector $(2,2,2,3,3)$ is isomorphic to the above item.
\end{classlist}

\setcounter{tabits}{0}
\begin{classlist}
\label{thm:classif-w23346}
Degree and geometric data of $\QQ$-factorial,
Gorenstein, Fano gtci threefolds of rank one,
type $(3,1)$ and weight vector $(2,3,3,4,6)$:
\begin{center}
\begin{longtable}{ccccc}
{\small ID} & {\small degree data}  & {\small $-\mathcal{K}$} & {\small $-\mathcal{K}^3$} & {\small $h^0(-\mathcal{K})$}
%
\input{latex-matrices-w23346t1.txt}
\end{longtable}  
\end{center}
\vspace*{-10pt}
Every $\QQ$-factorial Gorenstein Fano gtci threefold
of rank one, type $(3,1)$ and weight vector
$(2,3,3,4,6)$ is isomorphic to the above item.
\end{classlist}


\setcounter{tabits}{0}
\begin{classlist}
\label{thm:classif-w111111}
Degree and geometric data of $\QQ$-factorial,
Gorenstein, Fano gtci threefolds of rank one,
type $(3,2)$ and weight vector $(1,1,1,1,1,1)$:
\begin{center}
\begin{longtable}{ccccc}
{\small ID} & {\small degree data}  & {\small $-\mathcal{K}$} & {\small $-\mathcal{K}^3$} & {\small $h^0(-\mathcal{K})$}
%
\input{latex-matrices-w111111t1.txt}
%
\input{latex-matrices-w111111t2.txt}
%
\input{latex-matrices-w111111t22.txt}
\end{longtable}  
\end{center}
\vspace*{-10pt}
Every $\QQ$-factorial Gorenstein Fano gtci
threefold of rank one, type $(3,2)$ and weight
vector $(1,1,1,1,1,1)$ is isomorphic to 
precisely one of the \thetabits\ listed items.
\end{classlist}

\setcounter{tabits}{0}
\begin{classlist}
\label{thm:classif-w112222}
Degree and geometric data of $\QQ$-factorial,
Gorenstein, Fano gtci threefolds of rank one,
type $(3,2)$ and weight vector $(1,1,2,2,2,2)$:
\begin{center}
\begin{longtable}{ccccc}
{\small ID} & {\small degree data}  & {\small $-\mathcal{K}$} & {\small $-\mathcal{K}^3$} & {\small $h^0(-\mathcal{K})$}
%
\input{latex-matrices-w112222t1.txt}
%
\input{latex-matrices-w112222t2.txt}
\end{longtable}  
\end{center}
\vspace*{-10pt}
Every $\QQ$-factorial Gorenstein Fano gtci
threefold of rank one, type $(3,2)$ and weight
vector $(1,1,2,2,2,2)$ is isomorphic
to precisely one of the \thetabits\ listed items.
\end{classlist}

\setcounter{tabits}{0}
\begin{classlist}
\label{thm:classif-w123333}
Degree and geometric data of $\QQ$-factorial,
Gorenstein, Fano gtci threefolds of rank one,
type $(3,2)$ and weight vector $(1,2,3,3,3,3)$:
\begin{center}
\begin{longtable}{ccccc}
{\small ID} & {\small degree data}  & {\small $-\mathcal{K}$} & {\small $-\mathcal{K}^3$} & {\small $h^0(-\mathcal{K})$}
%
\input{latex-matrices-w123333t1.txt}
\end{longtable}  
\end{center}
\vspace*{-10pt}
Every $\QQ$-factorial Gorenstein Fano gtci
threefold of rank one, type $(3,2)$ and weight
vector $(1,2,3,3,3,3)$ is isomorphic to the above
item.
\end{classlist}

\setcounter{tabits}{0}
\begin{classlist}
\label{thm:classif-w222233}
Degree and geometric data of $\QQ$-factorial,
Gorenstein, Fano gtci threefolds of rank one,
type $(3,2)$ and weight vector $(2,2,2,2,3,3)$:
\begin{center}
\begin{longtable}{ccccc}
{\small ID} & {\small degree data}  & {\small $-\mathcal{K}$} & {\small $-\mathcal{K}^3$} & {\small $h^0(-\mathcal{K})$}
%
\input{latex-matrices-w222233t1.txt}
\end{longtable}  
\end{center}
\vspace*{-10pt}
Every $\QQ$-factorial Gorenstein Fano gtci
threefold of rank one, type $(3,2)$ and weight
vector $(2,2,2,2,3,3)$ is isomorphic to the above
item.
\end{classlist}


\setcounter{tabits}{0}
\begin{classlist}
\label{thm:classif-w1111111}
Degree and geometric data of $\QQ$-factorial,
Gorenstein, Fano gtci threefolds of rank one,
type $(3,3)$:
\begin{center}
\begin{longtable}{cc@{}ccc}
{\small ID} & {\small degree data}  & {\small $-\mathcal{K}$} & {\small $-\mathcal{K}^3$} & {\small $h^0(-\mathcal{K})$} 
%
\renewcommand{\arraystretch}{0}
\input{latex-matrices-w1111111t1.txt}
\renewcommand{\arraystretch}{0}
\input{latex-matrices-w1111111t2.txt}
\renewcommand{\arraystretch}{0}
\input{latex-matrices-w1111111t22.txt}
\end{longtable}  
\end{center}
\vspace*{-10pt}
Every $\QQ$-factorial Gorenstein Fano gtci threefold of
rank one and type $(3,3)$ is isomorphic to
precisely one of the \thetabits\ listed items.
\end{classlist}

\begin{remark}
The smooth gtci occurring in the above classification lists
are classically known; their ids are the following:
\begin{center}
{\small\rm w11111t1-1},
\quad
{\small\rm w11111t1-2},  
\quad
{\small\rm w11111t1-3},  
\quad
{\small\rm w11112t1-1},  
\quad
{\small\rm w11113t1-1},  
\\[3pt]
{\small\rm w11123t1-1},  
\qquad
{\small\rm w111111t1-1},  
\qquad
{\small\rm w111111t1-2},  
\qquad
{\small\rm w11111111t1-1}.  
\end{center}
Moreover, these gtci form exactly the overlap of our classification
with the terminal case provided by~\cite[Thm.~1.3]{HaMaWr},
and they all occur in~\cite[Thm.~4.1]{DHHKL}.
\end{remark}  

%% file: latex-matrices-w11111t1.txt
\\ \\
\stepcounter{tabits}
\stepcounter{idtci}
{\tiny\rm w11111t1-\theidtci}
&
{\small
$
\left[
\begin{array}{ccccc|c|c}

1 & 1 & 1 & 1 & 1 & \scriptstyle \ZZ & 2
\end{array}
\right]
$
}
&
{\small
$
\left[
\begin{array}{c}
3
\end{array}
\right]
$
}
&
$54$
&
$30$
\\ \\
\stepcounter{tabits}
\stepcounter{idtci}
{\tiny\rm w11111t1-\theidtci}
&
{\small
$
\left[
\begin{array}{ccccc|c|c}

1 & 1 & 1 & 1 & 1 & \scriptstyle \ZZ & 3
\end{array}
\right]
$
}
&
{\small
$
\left[
\begin{array}{c}
2
\end{array}
\right]
$
}
&
$24$
&
$15$
\\ \\
\stepcounter{tabits}
\stepcounter{idtci}
{\tiny\rm w11111t1-\theidtci}
&
{\small
$
\left[
\begin{array}{ccccc|c|c}

1 & 1 & 1 & 1 & 1 & \scriptstyle \ZZ & 4
\end{array}
\right]
$
}
&
{\small
$
\left[
\begin{array}{c}
1
\end{array}
\right]
$
}
&
$4$
&
$5$

%% file: latex-matrices-w11111t3.txt
\\ \\
\stepcounter{tabits}
\stepcounter{idtci}
{\tiny\rm w11111t3-\theidtci}
&
{\small
$
\left[
\begin{array}{ccccc|c|c}

1 & 1 & 1 & 1 & 1 & \scriptstyle \ZZ & 3
\\
\bar 0 &  \bar 0 &  \bar 0 &  \bar 1 &  \bar 2 & \scriptstyle \ZZ / 3\ZZ & \bar 0
\end{array}
\right]
$
}
&
{\small
$
\left[
\begin{array}{c}
2
\\
\bar 0
\end{array}
\right]
$
}
&
$8$
&
$7$

%% file: latex-matrices-w11112t1.txt
\\ \\
\stepcounter{tabits}
\setcounter{idtci}{0}
\stepcounter{idtci}
{\tiny\rm w11112t1-\theidtci}
&
{\small
$
\left[
\begin{array}{ccccc|c|c}

1 & 1 & 1 & 1 & 2 & \scriptstyle \ZZ & 4
\end{array}
\right]
$
}
&
{\small
$
\left[
\begin{array}{c}
2
\end{array}
\right]
$
}
&
$16$
&
$11$

%% file: latex-matrices-w11112t2.txt
\\ \\
\stepcounter{tabits}
\setcounter{idtci}{0}
\stepcounter{idtci}
{\tiny\rm w11112t2-\theidtci}
&
{\small
$
\left[
\begin{array}{ccccc|c|c}

1 & 1 & 1 & 1 & 2 & \scriptstyle \ZZ & 4
\\
\bar 0 &  \bar 0 &  \bar 0 &  \bar 1 &  \bar 1 & \scriptstyle \ZZ / 2\ZZ & \bar 0
\end{array}
\right]
$
}
&
{\small
$
\left[
\begin{array}{c}
2
\\
\bar 0
\end{array}
\right]
$
}
&
$8$
&
$7$
\\ \\
\stepcounter{tabits}
\stepcounter{idtci}
{\tiny\rm w11112t2-\theidtci}
&
{\small
$
\left[
\begin{array}{ccccc|c|c}

1 & 1 & 1 & 1 & 2 & \scriptstyle \ZZ & 4
\\
\bar 0 &  \bar 0 &  \bar 1 &  \bar 1 &  \bar 0 & \scriptstyle \ZZ / 2\ZZ & \bar 0
\end{array}
\right]
$
}
&
{\small
$
\left[
\begin{array}{c}
2
\\
\bar 0
\end{array}
\right]
$
}
&
$8$
&
$7$

%% file: latex-matrices-w11112t4.txt
\\ \\

%% file: latex-matrices-w11113t1.txt
\\ \\
\stepcounter{tabits}
\setcounter{idtci}{0}
\stepcounter{idtci}
{\tiny\rm w11113t1-\theidtci}
&
{\small
$
\left[
\begin{array}{ccccc|c|c}

1 & 1 & 1 & 1 & 3 & \scriptstyle \ZZ & 6
\end{array}
\right]
$
}
&
{\small
$
\left[
\begin{array}{c}
1
\end{array}
\right]
$
}
&
$2$
&
$4$

%% file: latex-matrices-w11123t1.txt
\\ \\
\stepcounter{tabits}
\setcounter{idtci}{0}
\stepcounter{idtci}
{\tiny\rm w11123t1-\theidtci}
&
{\small
$
\left[
\begin{array}{ccccc|c|c}

1 & 1 & 1 & 2 & 3 & \scriptstyle \ZZ & 6
\end{array}
\right]
$
}
&
{\small
$
\left[
\begin{array}{c}
2
\end{array}
\right]
$
}
&
$8$
&
$7$

%% file: latex-matrices-w11123t2.txt
\\ \\
\stepcounter{tabits}
\setcounter{idtci}{0}
\stepcounter{idtci}
{\tiny\rm w11123t2-\theidtci}
&
{\small
$
\left[
\begin{array}{ccccc|c|c}

1 & 1 & 1 & 2 & 3 & \scriptstyle \ZZ & 6
\\
\bar 0 &  \bar 0 &  \bar 1 &  \bar 0 &  \bar 1 & \scriptstyle \ZZ / 2\ZZ & \bar 0
\end{array}
\right]
$
}
&
{\small
$
\left[
\begin{array}{c}
2
\\
\bar 0
\end{array}
\right]
$
}
&
$4$
&
$5$

%% file: latex-matrices-w11123t22.txt
\\ \\
\stepcounter{tabits}
\setcounter{idtci}{0}
\stepcounter{idtci}
{\tiny\rm w11123t22-\theidtci}
&
{\small
$
\left[
\begin{array}{ccccc|c|c}

1 & 1 & 1 & 2 & 3 & \scriptstyle \ZZ & 6
\\
\bar 0 &  \bar 0 &  \bar 1 &  \bar 0 &  \bar 1 & \scriptstyle \ZZ / 2\ZZ & \bar 0
\\
\bar 0 &  \bar 1 &  \bar 0 &  \bar 0 &  \bar 1 & \scriptstyle \ZZ / 2\ZZ & \bar 0
\end{array}
\right]
$
}
&
{\small
$
\left[
\begin{array}{c}
2
\\
\bar 0
\\
\bar 0
\end{array}
\right]
$
}
&
$2$
&
$4$

%% file: latex-matrices-w11133t1.txt
\\ \\
\stepcounter{tabits}
\setcounter{idtci}{0}
\stepcounter{idtci}
{\tiny\rm w11133t1-\theidtci}
&
{\small
$
\left[
\begin{array}{ccccc|c|c}

1 & 1 & 1 & 3 & 3 & \scriptstyle \ZZ & 6
\end{array}
\right]
$
}
&
{\small
$
\left[
\begin{array}{c}
3
\end{array}
\right]
$
}
&
$18$
&
$12$

%% file: latex-matrices-w11133t3.txt
\\ \\
\stepcounter{tabits}
\setcounter{idtci}{0}
\stepcounter{idtci}
{\tiny\rm w11133t3-\theidtci}
&
{\small
$
\left[
\begin{array}{ccccc|c|c}

1 & 1 & 1 & 3 & 3 & \scriptstyle \ZZ & 6
\\
\bar 0 &  \bar 1 &  \bar 2 &  \bar 0 &  \bar 0 & \scriptstyle \ZZ / 3\ZZ & \bar 0
\end{array}
\right]
$
}
&
{\small
$
\left[
\begin{array}{c}
3
\\
\bar 0
\end{array}
\right]
$
}
&
$6$
&
$6$

%% file: latex-matrices-w11222t1.txt
\\ \\
\stepcounter{tabits}
\setcounter{idtci}{0}
\stepcounter{idtci}
{\tiny\rm w11222t1-\theidtci}
&
{\small
$
\left[
\begin{array}{ccccc|c|c}

1 & 1 & 2 & 2 & 2 & \scriptstyle \ZZ & 4
\end{array}
\right]
$
}
&
{\small
$
\left[
\begin{array}{c}
4
\end{array}
\right]
$
}
&
$27$
&
$19$
\\ \\
\stepcounter{tabits}
\stepcounter{idtci}
{\tiny\rm w11222t1-\theidtci}
&
{\small
$
\left[
\begin{array}{ccccc|c|c}

1 & 1 & 2 & 2 & 2 & \scriptstyle \ZZ & 6
\end{array}
\right]
$
}
&
{\small
$
\left[
\begin{array}{c}
2
\end{array}
\right]
$
}
&
$6$
&
$6$

%% file: latex-matrices-w11222t2.txt
\\ \\
\stepcounter{tabits}
\setcounter{idtci}{0}
\stepcounter{idtci}
{\tiny\rm w11222t2-\theidtci}
&
{\small
$
\left[
\begin{array}{ccccc|c|c}

1 & 1 & 2 & 2 & 2 & \scriptstyle \ZZ & 4
\\
\bar 0 &  \bar 0 &  \bar 0 &  \bar 1 &  \bar 1 & \scriptstyle \ZZ / 2\ZZ & \bar 0
\end{array}
\right]
$
}
&
{\small
$
\left[
\begin{array}{c}
4
\\
\bar 0
\end{array}
\right]
$
}
&
$16$
&
$11$
\\ \\

%% file: latex-matrices-w11224t1.txt
\\ \\
\stepcounter{tabits}
\setcounter{idtci}{0}
\stepcounter{idtci}
{\tiny\rm w11224t1-\theidtci}
&
{\small
$
\left[
\begin{array}{ccccc|c|c}

1 & 1 & 2 & 2 & 4 & \scriptstyle \ZZ & 8
\end{array}
\right]
$
}
&
{\small
$
\left[
\begin{array}{c}
2
\end{array}
\right]
$
}
&
$4$
&
$5$

%% file: latex-matrices-w11224t2.txt
\\ \\
\stepcounter{tabits}
\setcounter{idtci}{0}
\stepcounter{idtci}
{\tiny\rm w11224t2-\theidtci}
&
{\small
$
\left[
\begin{array}{ccccc|c|c}

1 & 1 & 2 & 2 & 4 & \scriptstyle \ZZ & 8
\\
\bar 0 &  \bar 0 &  \bar 0 &  \bar 1 &  \bar 1 & \scriptstyle \ZZ / 2\ZZ & \bar 0
\end{array}
\right]
$
}
&
{\small
$
\left[
\begin{array}{c}
2
\\
\bar 0
\end{array}
\right]
$
}
&
$2$
&
$4$
\\ \\
\stepcounter{tabits}
\stepcounter{idtci}
{\tiny\rm w11224t2-\theidtci}
&
{\small
$
\left[
\begin{array}{ccccc|c|c}

1 & 1 & 2 & 2 & 4 & \scriptstyle \ZZ & 8
\\
\bar 0 &  \bar 1 &  \bar 0 &  \bar 0 &  \bar 1 & \scriptstyle \ZZ / 2\ZZ & \bar 0
\end{array}
\right]
$
}
&
{\small
$
\left[
\begin{array}{c}
2
\\
\bar 0
\end{array}
\right]
$
}
&
$2$
&
$4$

%% file: latex-matrices-w11244t1.txt
\\ \\
\stepcounter{tabits}
\setcounter{idtci}{0}
\stepcounter{idtci}
{\tiny\rm w11244t1-\theidtci}
&
{\small
$
\left[
\begin{array}{ccccc|c|c}

1 & 1 & 2 & 4 & 4 & \scriptstyle \ZZ & 8
\end{array}
\right]
$
}
&
{\small
$
\left[
\begin{array}{c}
4
\end{array}
\right]
$
}
&
$16$
&
$11$

%% file: latex-matrices-w11244t2.txt
\\ \\
\stepcounter{tabits}
\setcounter{idtci}{0}
\stepcounter{idtci}
{\tiny\rm w11244t2-\theidtci}
&
{\small
$
\left[
\begin{array}{ccccc|c|c}

1 & 1 & 2 & 4 & 4 & \scriptstyle \ZZ & 8
\\
\bar 0 &  \bar 0 &  \bar 1 &  \bar 0 &  \bar 1 & \scriptstyle \ZZ / 2\ZZ & \bar 0
\end{array}
\right]
$
}
&
{\small
$
\left[
\begin{array}{c}
4
\\
\bar 0
\end{array}
\right]
$
}
&
$8$
&
$7$
\\ \\
\stepcounter{tabits}
\stepcounter{idtci}
{\tiny\rm w11244t2-\theidtci}
&
{\small
$
\left[
\begin{array}{ccccc|c|c}

1 & 1 & 2 & 4 & 4 & \scriptstyle \ZZ & 8
\\
\bar 0 &  \bar 1 &  \bar 0 &  \bar 0 &  \bar 1 & \scriptstyle \ZZ / 2\ZZ & \bar 0
\end{array}
\right]
$
}
&
{\small
$
\left[
\begin{array}{c}
4
\\
\bar 0
\end{array}
\right]
$
}
&
$8$
&
$7$
\\ \\

%% file: latex-matrices-w11246t1.txt
\\ \\
\stepcounter{tabits}
\setcounter{idtci}{0}
\stepcounter{idtci}
{\tiny\rm w11246t1-\theidtci}
&
{\small
$
\left[
\begin{array}{ccccc|c|c}

1 & 1 & 2 & 4 & 6 & \scriptstyle \ZZ & 12
\end{array}
\right]
$
}
&
{\small
$
\left[
\begin{array}{c}
2
\end{array}
\right]
$
}
&
$2$
&
$4$

%% file: latex-matrices-w11446t1.txt
\\ \\
\stepcounter{tabits}
\setcounter{idtci}{0}
\stepcounter{idtci}
{\tiny\rm w11446t1-\theidtci}
&
{\small
$
\left[
\begin{array}{ccccc|c|c}

1 & 1 & 4 & 4 & 6 & \scriptstyle \ZZ & 12
\end{array}
\right]
$
}
&
{\small
$
\left[
\begin{array}{c}
4
\end{array}
\right]
$
}
&
$8$
&
$7$

%% file: latex-matrices-w11446t2.txt
\\ \\
\stepcounter{tabits}
\setcounter{idtci}{0}
\stepcounter{idtci}
{\tiny\rm w11446t2-\theidtci}
&
{\small
$
\left[
\begin{array}{ccccc|c|c}

1 & 1 & 4 & 4 & 6 & \scriptstyle \ZZ & 12
\\
\bar 0 &  \bar 1 &  \bar 0 &  \bar 0 &  \bar 1 & \scriptstyle \ZZ / 2\ZZ & \bar 0
\end{array}
\right]
$
}
&
{\small
$
\left[
\begin{array}{c}
4
\\
\bar 0
\end{array}
\right]
$
}
&
$4$
&
$5$

%% file: latex-matrices-w11466t1.txt
\\ \\
\stepcounter{tabits}
\setcounter{idtci}{0}
\stepcounter{idtci}
{\tiny\rm w11466t1-\theidtci}
&
{\small
$
\left[
\begin{array}{ccccc|c|c}

1 & 1 & 4 & 6 & 6 & \scriptstyle \ZZ & 12
\end{array}
\right]
$
}
&
{\small
$
\left[
\begin{array}{c}
6
\end{array}
\right]
$
}
&
$18$
&
$12$

%% file: latex-matrices-w11466t3.txt
\\ \\
\stepcounter{tabits}
\setcounter{idtci}{0}
\stepcounter{idtci}
{\tiny\rm w11466t3-\theidtci}
&
{\small
$
\left[
\begin{array}{ccccc|c|c}

1 & 1 & 4 & 6 & 6 & \scriptstyle \ZZ & 12
\\
\bar 0 &  \bar 1 &  \bar 2 &  \bar 0 &  \bar 0 & \scriptstyle \ZZ / 3\ZZ & \bar 0
\end{array}
\right]
$
}
&
{\small
$
\left[
\begin{array}{c}
6
\\
\bar 0
\end{array}
\right]
$
}
&
$6$
&
$6$

%% file: latex-matrices-w12223t1.txt
\\ \\
\stepcounter{tabits}
\setcounter{idtci}{0}
\stepcounter{idtci}
{\tiny\rm w12223t1-\theidtci}
&
{\small
$
\left[
\begin{array}{ccccc|c|c}

1 & 2 & 2 & 2 & 3 & \scriptstyle \ZZ & 6
\end{array}
\right]
$
}
&
{\small
$
\left[
\begin{array}{c}
4
\end{array}
\right]
$
}
&
$16$
&
$11$

%% file: latex-matrices-w12225t1.txt
\\ \\
\stepcounter{tabits}
\setcounter{idtci}{0}
\stepcounter{idtci}
{\tiny\rm w12225t1-\theidtci}
&
{\small
$
\left[
\begin{array}{ccccc|c|c}

1 & 2 & 2 & 2 & 5 & \scriptstyle \ZZ & 10
\end{array}
\right]
$
}
&
{\small
$
\left[
\begin{array}{c}
2
\end{array}
\right]
$
}
&
$2$
&
$4$

%% file: latex-matrices-w12333t1.txt
\\ \\
\stepcounter{tabits}
\setcounter{idtci}{0}
\stepcounter{idtci}
{\tiny\rm w12333t1-\theidtci}
&
{\small
$
\left[
\begin{array}{ccccc|c|c}

1 & 2 & 3 & 3 & 3 & \scriptstyle \ZZ & 6
\end{array}
\right]
$
}
&
{\small
$
\left[
\begin{array}{c}
6
\end{array}
\right]
$
}
&
$24$
&
$15$

%% file: latex-matrices-w12333t2.txt
\\ \\
\stepcounter{tabits}
\setcounter{idtci}{0}
\stepcounter{idtci}
{\tiny\rm w12333t2-\theidtci}
&
{\small
$
\left[
\begin{array}{ccccc|c|c}

1 & 2 & 3 & 3 & 3 & \scriptstyle \ZZ & 6
\\
\bar 0 &  \bar 0 &  \bar 0 &  \bar 1 &  \bar 1 & \scriptstyle \ZZ / 2\ZZ & \bar 0
\end{array}
\right]
$
}
&
{\small
$
\left[
\begin{array}{c}
6
\\
\bar 0
\end{array}
\right]
$
}
&
$12$
&
$9$

%% file: latex-matrices-w12333t22.txt
\\ \\
\stepcounter{tabits}
\setcounter{idtci}{0}
\stepcounter{idtci}
{\tiny\rm w12333t22-\theidtci}
&
{\small
$
\left[
\begin{array}{ccccc|c|c}

1 & 2 & 3 & 3 & 3 & \scriptstyle \ZZ & 6
\\
\bar 0 &  \bar 0 &  \bar 0 &  \bar 1 &  \bar 1 & \scriptstyle \ZZ / 2\ZZ & \bar 0
\\
\bar 0 &  \bar 0 &  \bar 1 &  \bar 0 &  \bar 1 & \scriptstyle \ZZ / 2\ZZ & \bar 0
\end{array}
\right]
$
}
&
{\small
$
\left[
\begin{array}{c}
6
\\
\bar 0
\\
\bar 0
\end{array}
\right]
$
}
&
$6$
&
$6$

%% file: latex-matrices-w12366t1.txt
\\ \\
\stepcounter{tabits}
\setcounter{idtci}{0}
\stepcounter{idtci}
{\tiny\rm w12366t1-\theidtci}
&
{\small
$
\left[
\begin{array}{ccccc|c|c}

1 & 2 & 3 & 6 & 6 & \scriptstyle \ZZ & 12
\end{array}
\right]
$
}
&
{\small
$
\left[
\begin{array}{c}
6
\end{array}
\right]
$
}
&
$12$
&
$9$

%% file: latex-matrices-w12366t2.txt
\\ \\
\stepcounter{tabits}
\setcounter{idtci}{0}
\stepcounter{idtci}
{\tiny\rm w12366t2-\theidtci}
&
{\small
$
\left[
\begin{array}{ccccc|c|c}

1 & 2 & 3 & 6 & 6 & \scriptstyle \ZZ & 12
\\
\bar 0 &  \bar 0 &  \bar 1 &  \bar 0 &  \bar 1 & \scriptstyle \ZZ / 2\ZZ & \bar 0
\end{array}
\right]
$
}
&
{\small
$
\left[
\begin{array}{c}
6
\\
\bar 0
\end{array}
\right]
$
}
&
$6$
&
$6$
\\ \\
\stepcounter{tabits}
\stepcounter{idtci}
{\tiny\rm w12366t2-\theidtci}
&
{\small
$
\left[
\begin{array}{ccccc|c|c}

1 & 2 & 3 & 6 & 6 & \scriptstyle \ZZ & 12
\\
\bar 0 &  \bar 1 &  \bar 1 &  \bar 0 &  \bar 0 & \scriptstyle \ZZ / 2\ZZ & \bar 0
\end{array}
\right]
$
}
&
{\small
$
\left[
\begin{array}{c}
6
\\
\bar 0
\end{array}
\right]
$
}
&
$6$
&
$6$

%% file: latex-matrices-w12669t1.txt
\\ \\
\stepcounter{tabits}
\stepcounter{idtci}
{\tiny\rm w12669t1-\theidtci}
&
{\small
$
\left[
\begin{array}{ccccc|c|c}

1 & 2 & 6 & 6 & 9 & \scriptstyle \ZZ & 18
\end{array}
\right]
$
}
&
{\small
$
\left[
\begin{array}{c}
6
\end{array}
\right]
$
}
&
$6$
&
$6$

%% file: latex-matrices-w13444t1.txt
\\ \\
\stepcounter{tabits}
\setcounter{idtci}{0}
\stepcounter{idtci}
{\tiny\rm w13444t1-\theidtci}
&
{\small
$
\left[
\begin{array}{ccccc|c|c}

1 & 3 & 4 & 4 & 4 & \scriptstyle \ZZ & 12
\end{array}
\right]
$
}
&
{\small
$
\left[
\begin{array}{c}
4
\end{array}
\right]
$
}
&
$4$
&
$5$

%% file: latex-matrices-w138CCt1.txt
\\ \\
\stepcounter{tabits}
\setcounter{idtci}{0}
\stepcounter{idtci}
{\tiny\rm w138CCt1-\theidtci}
&
{\small
$
\left[
\begin{array}{ccccc|c|c}

1 & 3 & 8 & 12 & 12 & \scriptstyle \ZZ & 24
\end{array}
\right]
$
}
&
{\small
$
\left[
\begin{array}{c}
12
\end{array}
\right]
$
}
&
$12$
&
$9$

%% file: latex-matrices-w145AAt1.txt
\\ \\
\stepcounter{tabits}
\setcounter{idtci}{0}
\stepcounter{idtci}
{\tiny\rm w145AAt1-\theidtci}
&
{\small
$
\left[
\begin{array}{ccccc|c|c}

1 & 4 & 5 & 10 & 10 & \scriptstyle \ZZ & 20
\end{array}
\right]
$
}
&
{\small
$
\left[
\begin{array}{c}
10
\end{array}
\right]
$
}
&
$10$
&
$8$

%% file: latex-matrices-w22233t1.txt
\\ \\
\stepcounter{tabits}
\stepcounter{idtci}
{\tiny\rm w22233t1-\theidtci}
&
{\small
$
\left[
\begin{array}{ccccc|c|c}

2 & 2 & 2 & 3 & 3 & \scriptstyle \ZZ & 6
\end{array}
\right]
$
}
&
{\small
$
\left[
\begin{array}{c}
6
\end{array}
\right]
$
}
&
$18$
&
$12$

%% file: latex-matrices-w23346t1.txt
\\ \\
\stepcounter{tabits}
\setcounter{idtci}{0}
\stepcounter{idtci}
{\tiny\rm w23346t1-\theidtci}
&
{\small
$
\left[
\begin{array}{ccccc|c|c}

2 & 3 & 3 & 4 & 6 & \scriptstyle \ZZ & 12
\end{array}
\right]
$
}
&
{\small
$
\left[
\begin{array}{c}
6
\end{array}
\right]
$
}
&
$6$
&
$6$

%% file: latex-matrices-w111111t1.txt
\\ \\
\stepcounter{tabits}
\setcounter{idtci}{0}
\stepcounter{idtci}
{\tiny\rm w111111t1-\theidtci}
&
{\small
$
\left[
\begin{array}{cccccc|c|cc}

1 & 1 & 1 & 1 & 1 & 1 & \scriptstyle \ZZ & 2 & 2
\end{array}
\right]
$
}
&
{\small
$
\left[
\begin{array}{c}
2
\end{array}
\right]
$
}
&
$32$
&
$19$
\\ \\
\stepcounter{tabits}
\stepcounter{idtci}
{\tiny\rm w111111t1-\theidtci}
&
{\small
$
\left[
\begin{array}{cccccc|c|cc}

1 & 1 & 1 & 1 & 1 & 1 & \scriptstyle \ZZ & 2 & 3
\end{array}
\right]
$
}
&
{\small
$
\left[
\begin{array}{c}
1
\end{array}
\right]
$
}
&
$6$
&
$6$

%% file: latex-matrices-w111111t2.txt
\\ \\
\stepcounter{tabits}
\setcounter{idtci}{0}
\stepcounter{idtci}
{\tiny\rm w111111t2-\theidtci}
&
{\small
$
\left[
\begin{array}{cccccc|c|cc}

1 & 1 & 1 & 1 & 1 & 1 & \scriptstyle \ZZ & 2 & 2
\\
\bar 0 &  \bar 0 &  \bar 0 &  \bar 0 &  \bar 1 &  \bar 1 & \scriptstyle \ZZ / 2\ZZ & \bar 0 & \bar 0
\end{array}
\right]
$
}
&
{\small
$
\left[
\begin{array}{c}
2
\\
\bar 0
\end{array}
\right]
$
}
&
$16$
&
$11$

%% file: latex-matrices-w111111t22.txt
\\ \\

%% file: latex-matrices-w112222t1.txt
\\ \\
\stepcounter{tabits}
\setcounter{idtci}{0}
\stepcounter{idtci}
{\tiny\rm w112222t1-\theidtci}
&
{\small
$
\left[
\begin{array}{cccccc|c|cc}

1 & 1 & 2 & 2 & 2 & 2 & \scriptstyle \ZZ & 4 & 4
\end{array}
\right]
$
}
&
{\small
$
\left[
\begin{array}{c}
2
\end{array}
\right]
$
}
&
$8$
&
$7$

%% file: latex-matrices-w112222t2.txt
\\ \\
\stepcounter{tabits}
\setcounter{idtci}{0}
\stepcounter{idtci}
{\tiny\rm w112222t2-\theidtci}
&
{\small
$
\left[
\begin{array}{cccccc|c|cc}

1 & 1 & 2 & 2 & 2 & 2 & \scriptstyle \ZZ & 4 & 4
\\
\bar 0 &  \bar 0 &  \bar 0 &  \bar 0 &  \bar 1 &  \bar 1 & \scriptstyle \ZZ / 2\ZZ & \bar 0 & \bar 0
\end{array}
\right]
$
}
&
{\small
$
\left[
\begin{array}{c}
2
\\
\bar 0
\end{array}
\right]
$
}
&
$4$
&
$5$
\\ \\
\stepcounter{tabits}
\stepcounter{idtci}
{\tiny\rm w112222t2-\theidtci}
&
{\small
$
\left[
\begin{array}{cccccc|c|cc}

1 & 1 & 2 & 2 & 2 & 2 & \scriptstyle \ZZ & 4 & 4
\\
\bar 0 &  \bar 1 &  \bar 0 &  \bar 0 &  \bar 0 &  \bar 1 & \scriptstyle \ZZ / 2\ZZ & \bar 0 & \bar 0
\end{array}
\right]
$
}
&
{\small
$
\left[
\begin{array}{c}
2
\\
\bar 0
\end{array}
\right]
$
}
&
$4$
&
$5$
\\ \\

%% file: latex-matrices-w123333t1.txt
\\ \\
\stepcounter{tabits}
\setcounter{idtci}{0}
\stepcounter{idtci}
{\tiny\rm w123333t1-\theidtci}
&
{\small
$
\left[
\begin{array}{cccccc|c|cc}

1 & 2 & 3 & 3 & 3 & 3 & \scriptstyle \ZZ & 6 & 6
\end{array}
\right]
$
}
&
{\small
$
\left[
\begin{array}{c}
3
\end{array}
\right]
$
}
&
$6$
&
$6$

%% file: latex-matrices-w222233t1.txt
\\ \\
\stepcounter{tabits}
\setcounter{idtci}{0}
\stepcounter{idtci}
{\tiny\rm w222233t1-\theidtci}
&
{\small
$
\left[
\begin{array}{cccccc|c|cc}

2 & 2 & 2 & 2 & 3 & 3 & \scriptstyle \ZZ & 6 & 6
\end{array}
\right]
$
}
&
{\small
$
\left[
\begin{array}{c}
2
\end{array}
\right]
$
}
&
$2$
&
$4$

%% file: latex-matrices-w1111111t1.txt
\\ \\
\stepcounter{tabits}
\setcounter{idtci}{0}
\stepcounter{idtci}
{\tiny\rm w1111111t1-\theidtci}
&
{\small
$
\left[
\begin{array}{ccccccc|c|ccc}

1 & 1 & 1 & 1 & 1 & 1 & 1 & \scriptstyle \ZZ & 2 & 2 & 2
\end{array}
\right]
$
}
&
{\small
$
\left[
\begin{array}{c}
1
\end{array}
\right]
$
}
&
$8$
&
$7$

%% file: latex-matrices-w1111111t2.txt
\\ \\
\stepcounter{tabits}
\setcounter{idtci}{0}
\stepcounter{idtci}
{\tiny\rm w1111111t2-\theidtci}
&
{\small
$
\left[
\begin{array}{ccccccc|c|ccc}

1 & 1 & 1 & 1 & 1 & 1 & 1 & \scriptstyle \ZZ & 2 & 2 & 2
\\
\bar 0 &  \bar 0 &  \bar 0 &  \bar 0 &  \bar 0 &  \bar 1 &  \bar 1 & \scriptstyle \ZZ / 2\ZZ & \bar 0 & \bar 0 & \bar 0
\end{array}
\right]
$
}
&
{\small
$
\left[
\begin{array}{c}
1
\\
\bar 0
\end{array}
\right]
$
}
&
$4$
&
$5$

%% file: latex-matrices-w1111111t22.txt
\\ \\
\stepcounter{tabits}
\setcounter{idtci}{0}
\stepcounter{idtci}
{\tiny\rm w1111111t22-\theidtci}
&
{\small
$
\left[
\begin{array}{ccccccc|c|ccc}

1 & 1 & 1 & 1 & 1 & 1 & 1 & \scriptstyle \ZZ & 2 & 2 & 2
\\
\bar 0 &  \bar 0 &  \bar 0 &  \bar 0 &  \bar 0 &  \bar 1 &  \bar 1 & \scriptstyle \ZZ / 2\ZZ & \bar 0 & \bar 0 & \bar 0
\\
\bar 0 &  \bar 0 &  \bar 0 &  \bar 0 &  \bar 1 &  \bar 0 &  \bar 1 & \scriptstyle \ZZ / 2\ZZ & \bar 0 & \bar 0 & \bar 0
\end{array}
\right]
$
}
&
{\small
$
\left[
\begin{array}{c}
1
\\
\bar 0
\\
\bar 0
\end{array}
\right]
$
}
&
$2$
&
$4$